\documentclass[11pt]{article}%
\usepackage{float} % fix table inside context 
\usepackage{fullpage}
\usepackage{amsfonts}
\usepackage{amsmath}
\usepackage{amssymb}
\usepackage{amsbsy}
\usepackage{amsthm}
\usepackage{thmtools} % This helps cleveref recognize environments
\usepackage{cleveref}
\usepackage{mathtools}
\usepackage{epsfig}
\usepackage{graphicx}
\usepackage{colordvi}
\usepackage{graphics}
\usepackage{color}
\usepackage{bm}
\usepackage{verbatim} 
\numberwithin{equation}{section}

\usepackage[numbers,sort]{natbib}

\newtheorem{theorem}{Theorem}[section]
\newtheorem{thm}{Theorem}[section]

\newtheorem{alg}[thm]{Algorithm}
\newtheorem{remark}[thm]{Remark}

\begin{document}

\title{Continuous-data-assimilation-enabled fast and robust convergence of an Uzawa-based solver for Navier-Stokes equations with large Reynolds number}

%\title{Data-assimilation-enabled and -accelerated convergence of a fast Uzawa-based solver for Navier-Stokes equations with large Reynolds number}

\author{
Victoria Luongo Fisher \thanks{\small School of Mathematical and Statistical Sciences, Clemson University, Clemson, SC, 29364. Partially supported by Department of Energy grant DE-SC0025292 (vluongo@clemson.edu)}
\and
Jessica C. Franklin \thanks{\small School of Mathematical and Statistical Sciences, Clemson University, Clemson, SC, 29364. Partially supported by Department of Energy grant DE-SC0025292 (jfrank8@clemson.edu)}
\and
Leo G. Rebholz\thanks{\small School of Mathematical and Statistical Sciences, Clemson University, Clemson, SC, 29364.  Partially supported by Department of Energy grant DE-SC0025292  (rebholz@clemson.edu)}
}
\maketitle

\begin{abstract}  
This paper shows how continuous data assimilation (CDA) can be used to provably enable and accelerate convergence of a (efficient at each iteration due to a physics-splitting, but generally slowly converging and not robust) nonlinear solver  for incompressible Navier-Stokes equations (NSE).  Herein we develop, analyze and test an Uzawa-based nonlinear solver for incompressible NSE that incorporates partial solution data into the iteration through continuous data assimilation (CDA-Uzawa).  We rigorously prove that i) CDA-Uzawa will accelerate a converging Uzawa iteration, and more partial solution data yields more acceleration, and ii) with enough partial solution data CDA-Uzawa will converge for arbitrarily large Reynolds numbers, even if multiple NSE solutions exist. In the case of noisy data, we prove that the convergence results hold down to the size of the noise, and we propose a strategy to pass to Newton once CDA-Uzawa convergence reaches its lower limit.  Results of several numerical tests illustrate the theory and show CDA-Uzawa is a very effective and efficient solver.   While this paper focuses a particular splitting-based solver for the NSE, the key ideas are quite general and extendable to a wide class of nonlinear solvers.  
\end{abstract}

\section{Introduction}

Simulating fluid flow is an important subtask in many applications across science and engineering.  For incompressible, Newtonian fluids such as water and oil, this is done by solving the incompressible Navier-Stokes equations (NSE) which are given on $\Omega\times (0,T],$ where $\Omega \subset \mathbb{R}^d$ ($d =2, 3$), by
\begin{align}
   \label{eq1} u_t + u \cdot \nabla u + \nabla p - \nu \Delta u &= f,\\
   \label{eq2} \nabla \cdot u &= 0, \\
   \label{eq3} u |_{\partial \Omega} &= 0, \\
   \label{eq4b} u(x,0) & = u_0(x),
\end{align}
where $u$ is the velocity, $p$ the pressure, $f$ an external forcing, $\nu$ the kinematic viscosity, and 
$u_0$ is the initial condition. Our analysis considers for simplicity the case of homogeneous Dirichlet boundary conditions, but our results can be extended to nonhomogeneous mixed Dirichlet/Neumann boundary conditions without significant difficulties.  We will focus on nonlinear solvers in this paper, and so restrict to the steady case, i.e. where $u_t=0$.  However, our results are easily extendable to the nonlinear problems that arise at each time step of a temporal discretization of the time dependent case (by replacing $u_t$ in \eqref{eq1} by $\sigma u$ where $\sigma \sim \Delta t^{-1}$ and adjusting the known right hand side appropriately).  Solving the steady NSE system is well-known to be challenging and become more difficult as the Reynolds number $Re\sim \nu^{-1}$ increases.  Even though time dependent problems have greater challenges in resolving fine level physics compared to the steady case, time dependent solvers often avoid issues with nonlinear solves by linearizing in time or by having good initial guesses (i.e. the solution at the last time step).

Perhaps the most commonly used solver for the NSE is the Picard iteration, sometimes also called Oseen iteration, which is defined by
\begin{align}
   \label{pic1} u_{k} \cdot \nabla u_{k+1} + \nabla p_{k+1} - \nu \Delta u_{k+1} &= f,\\
   \label{pic2} \nabla \cdot u_{k+1} &= 0.
\end{align}
Picard is known to converge linearly for any initial guess $u_0\in H^1(\Omega)$ provided the problem parameters are sufficiently small.  That is, if
\[
\alpha := M\nu^{-2} \| f\|_{H^{-1}} <1,
\]
where $M$ is a constant depending only on the domain size (see section 2 for a precise definition of $M$) \cite{GR86,temam}, then the linear convergence rate of Picard will be $\alpha$ \cite{GR86,temam}. However, for larger $\alpha$ (and thus equivalently, larger $Re$), Picard will usually fail to converge \cite{PRX19}.  Note also that for sufficiently large $\alpha$, the steady NSE are known to admit multiple solutions \cite{Laytonbook} which creates additional numerical difficulties.

To improve the performance of the Picard solver in the setting where partial solution data is available \footnote{Partial solution data can be obtained from, e.g.,  physical measurements or observations, or also if the holder of a very large scale computed solution is passing it to another interested party but can passes some of the solution due to its size.}, it was proposed in \cite{LHRV23} to incorporate the data into the nonlinear solver using continuous data assimilation (CDA) by modifying the Picard scheme via
	\begin{align*}
		-\nu \Delta u_{k+1}+u_{k}\cdot\nabla u_{k+1}+ \nabla p_{k+1} + \mu I_H(u_{k+1} - u)&={f}, \\
		\nabla\cdot {u}_{k+1}&=0.
	\end{align*}
Here, $\mu>0$ is the nudging (or penalty) parameter, and $I_H(u)$ is an appropriate interpolant of a NSE velocity solution $u$ with characteristic point spacing $H$. The true solution interpolant $I_H(u)$ is considered known since we assume that we know partial solution data by being given $u(x_i)$ for $i=1,2,...,N$  (more discussion of $I_H$ is given in Section 2).  It is proven in \cite{LHRV23} that for any $\alpha>0$, if $H\alpha^{1/2} \stackrel{\sim}{<} 1$ then CDA-Picard is guaranteed to converge linearly to the NSE solution from which the partial solution data was taken with rate $O(H^{1/2} \alpha)$.  Thus, if Picard is converging, then using CDA to incorporate partial solution data will accelerate the convergence; if Picard is not converging, then with enough solution data, CDA can enable Picard to converge.  The case of noisy data in CDA-Picard is studied in \cite{GLNR24}, and using CDA to improve the Newton iteration for NSE is considered in \cite{LHRV23}.  To improve efficiency of CDA-Picard, an incremental Picard Yosida (IPY) method was proposed for use with CDA in \cite{FRV25}, which provided similar robustness and acceleration as CDA-Picard but  by its construction CDA-IPY has an efficiency gain due to a more easily solvable Schur complement in the linear system solves at each iteration.

We note the CDA-Picard papers \cite{LHRV23,GLNR24} from 2024 were the first works that brought CDA ideas to bear for accelerating and enabling convergence of nonlinear solvers for steady problems (or at a time step of a transient problem).  CDA was originally developed by Azouani, Olson and Titi in 2014 to improve accuracy and stability of time dependent problems by nudging an interpolant of a computed solution toward a true solution interpolant at each time step \cite{AOT14,AT14}, leading to solutions that become more accurate as time evolves in a simulation.  These initial CDA papers of Azouani et al. led to an explosion of interest and success, and now hundreds of papers have been written to extend and improve the CDA approach to data assimilation for time dependent PDEs.  We refer readers interested in CDA to the recent papers \cite{FLMW24,JP23,LP24,ARR_2022,CFLS25,AJS_2022,CH22} and references therein.

While the acceleration and robustness improvements that CDA provides to the Picard nonlinear solver are outstanding, a fundamental drawback of CDA-Picard (and CDA-IPY) is that solving the saddle point linear systems that arise at each nonlinear iteration can be difficult and computationally expensive, especially as $Re$ increases.   The purpose of this paper is to study a significantly more efficient CDA-Uzawa nonlinear solver for solving the NSE when partial solution data is known.

Uzawa solvers are classical splitting schemes that require no Schur complement solves, and thus they are much cheaper at a single iteration compared to Picard and IPY due to an explicit splitting of the mass and momentum equations (the pressure in the momentum equation is lagged) \cite{CHS15,temam}.  However, it is also known that in general they do not converge nearly as fast as Picard in terms of total iterations, which makes the overall cost more expensive \cite{GRVZ23}.  Hence, their use as stand alone nonlinear solvers has long been abandoned, although some improvements in their analysis have recently been discovered \cite{CHS17,CHS15}.  It has recently been shown that by using grad-div stabilization and treating the viscous term implicitly, significant improvement in convergence can be obtained \cite{GRVZ23, TCEK23}.  Therefore, we study CDA applied to a grad-div stabilized Uzawa scheme, which takes the form (and which we refer to as CDA-Uzawa):
\begin{align}
	-\nu \Delta u_{k+1}+u_{k}\cdot\nabla u_{k+1}+ \nabla p_{k} - \gamma\nabla (\nabla \cdot u_{k+1}) + \mu I_H(u_{k+1} - u)&={f}, \label{cdau11} \\
		\gamma \nabla\cdot {u}_{k+1} + (p_{k+1} - p_k) &=0. \label{cdau12}
\end{align}
Note that \eqref{cdau11} can be solved independently, and then \eqref{cdau12} is used to recover $p_{k+1}$.  Further, if $\nabla \cdot u_{k+1}$ is in the pressure space (e.g. if Scott-Vogelius elements are used), then recovering $p_{k+1}$ is simply a calculation.  Hence at each iteration of CDA-Uzawa, the main cost is a single linear solve for \eqref{cdau11}; for comparison, CDA-Picard must solve a non-symmetric saddle point linear system that when solved with robust methods such as augmented-Lagrangian preconditioning approaches of \cite{benzi,HR13,BB12} requires 8-20 such linear solves (
2 linear solves very similar to \eqref{cdau11} at each outer iteration of GMRES, and each outer GMRES iteration typically requires 4-10 (or more) iterations to converge).  

%,  this means {\it CDA-Picard requires 8-20 linear solves of the form \eqref{cdau11} at each nonlinear iteration compared to just 1 required for CDA-Uzawa.} 

We show herein analytically and numerically that in a discrete setting where the divergence of the velocity space is a subspace of the pressure space (e.g. in the Scott-Vogelius element setting \cite{GS19,scott:vogelius:conforming}) and provided enough partial solution data, CDA-Uzawa 
has essentially the same convergence rate and convergence basin as CDA-Picard and CDA-IPY ({\it in other words, CDA makes it so there is no cost in convergence for the gain in efficiency at each iteration!}).  We will show CDA-Uzawa solver's associated Lipschitz constant is $\sim H^{1/2}\alpha$, just like it is for CDA-Picard and CDA-IPY.  
%However, at each iteration, both CDA-Picard and CDA-IPY require two solves that are analogous to the single solve needed by CDA-Uzawa in \eqref{cdau11}, and worse, CDA-Picard has a nonsymmetric Schur complement to solve while CDA-IPY has a Stokes Schur complement to solve.  
This is quite surprising in our view, since Uzawa is generally a significantly worse solver than Picard in terms of convergence rate and robustness with respect to the problem parameters $Re$ and $\alpha$, but CDA essentially fixes these problems with Uzawa.
%, and yet CDA-Uzawa has the same Lipschitz constant as CDA-Picard and thus yields about the same rate and about the same robustness.  
Often in scientific computing there is a trade-off of efficiency versus accuracy/effectiveness, but in this case CDA-Uzawa gains significantly in efficiency over CDA-Picard but maintains the same convergence properties.

This paper is arranged as follows. Section 2 provides the mathematical preliminaries needed for a smooth analysis to follow. In Section 3, we detail the convergence theory of CDA-Uzawa in both the noise-free data case and the noisy partial solution data setting. Section 4 gives the numerical results to illustrate the CDA-Uzawa theory. In Section 5, we draw conclusions and present future directions.

\section{Mathematical Preliminaries}
\indent We consider $\Omega \subset \mathbb{R}^d$, $d = 2, 3$, to be a regular domain, and denote the $L^2(\Omega)$ norm by $|| \cdot ||$ and inner product by $(\cdot, \cdot)$. All other norms will be labeled with subscripts.  Define the natural NSE velocity and pressure spaces to be
\begin{align*}
    X & := H_0^1(\Omega) = \{ v\in H^1(\Omega),\ v|_{\partial\Omega}=0 \}, \\
    Q & := L^2_0(\Omega) = \{ q\in L^2(\Omega),\ \int_{\Omega} q=0 \}.
\end{align*}
Define the divergence-free velocity space by
\begin{equation*}
    V := \{v \in X \ : \ (\nabla \cdot v, q) = 0, \ \forall q \in Q \}.
\end{equation*}

Recall that the Poincar\' e-Friedrichs inequality holds in $X$: there exists a constant $C_P > 0$ depending only on the size of $\Omega$ such that for every $v \in X$,
\begin{equation}
    \label{prelimeq7} ||v|| \leq C_P||\nabla v||.
\end{equation}
The dual space of $X$ will be denoted as $X'$, and the associated norm of $X'$ will be denoted as $||\cdot ||_{-1}$.  The notation $\langle \cdot, \cdot \rangle$ is used for the dual pairing of $X$ and $X'$.

The following bound is important in our analysis and is well known from e.g. \cite{Laytonbook}: there exists a constant $M$ that is only dependent on the diameter of $\Omega$ such that for any $\chi\in V$ and $v,w\in X$,
\begin{align}
    \label{bbound} (\chi\cdot\nabla  w, v) &\leq M ||\chi||^{\frac{1}{2}} ||\nabla \chi ||^{\frac{1}{2}} ||\nabla w|| ||\nabla v||.
\end{align}

\subsection{NSE Preliminaries}

The weak form of the steady NSE is as follows: find $(u, p) \in (X, Q)$ satisfying for all $(v, q) \in (X, Q)$, 
\begin{align}
   \label{eq4} \nu (\nabla u, \nabla v) + b(u, u, v) - (p, \nabla \cdot v)  &= \langle f, v \rangle,\\
   \label{eq5} (\nabla \cdot u, q) &= 0.
\end{align}

The system \eqref{eq1}-\eqref{eq3} is known to admit solutions for any problem data satisfying $f\in H^{-1}(\Omega)$ and $\nu>0$  \cite{Laytonbook}, and any such solution is bounded by
\begin{equation}
    \label{prelimeq3} ||\nabla u|| \leq \nu^{-1} ||f||_{-1}.
\end{equation}
Additionally, if the (sufficient) condition holds on the size of the problem parameters
\[
\alpha := M \nu^{-2} ||f||_{-1}<1,
\]
then the solutions are also unique \cite{Laytonbook,temam}.  For sufficiently large $\alpha$, it is known that multiple solutions to \eqref{eq4}-\eqref{eq5} may exist \cite{Laytonbook}.

\begin{remark}
Our analysis avoids smallness assumptions on $\alpha$, and thus on the Reynolds number as well.  In the case of $\alpha$ that is large enough for multiple solutions to exist, our analysis proves convergence of CDA-Uzawa to the solution from which the partial solution data $I_H(u)$ was obtained.
\end{remark}

\subsection{Finite Element Preliminaries}

Let $\tau_h$ be a regular, conforming triangulation of $\Omega$, with $h$ denoting the maximum element diameter.  Denote by $P_k$ the space of degree $k \geq 1$ piecewise continuous polynomials on $\tau_h$. $P_k^{disc}$ will denote the space of degree $k \geq 0$ piecewise discontinuous polynomials on $\tau_h$. 

We use discrete velocity-pressure spaces $(X_h, Q_h) \subset (X, Q)$ such that the LBB condition holds: there exists a constant $\beta$ that is independent of $h$ and satisfies
\begin{equation}
    \label{prelimeq6} \inf_{q \in Q_h} \sup_{v \in X_h} \frac{(\nabla \cdot v, q)}{||q||||\nabla v||} \geq \beta > 0.
\end{equation}
We also assume that $\nabla \cdot X_h \subseteq Q_h$, which is a critical assumption for the success of the CDA-Uzawa scheme.  Elements that satisfy these assumptions include  $(P_k,P_{k-1}^{disc})$ Scott-Vogelius elements on meshes with appropriate macro-element structures \cite{arnold:qin:scott:vogelius:2D,GS19,JLMNR17,scott:vogelius:conforming, zhang:scott:vogelius:3D}.

Because of our assumption that $\nabla \cdot X_h \subseteq Q_h$, there is no difference between our analysis done on $(X,Q)$ versus $(X_h,Q_h)$.  Hence, we will do our analysis simply on $(X,Q)$ to save on notation, and our numerical tests use Scott-Vogelius finite elements.

\subsection{Uzawa}

Uzawa schemes are classical efficient splitting schemes for Stokes and  NSE \cite{BGHRR25,CHS15,MS18,temam}, and there are many minor variations of them for NSE including e.g. different linearization of the convective term in the iteration \cite{CHS15}, stabilizing the convergence (which is sometimes renamed to Arrow-Hurwicz) \cite{CHS17, temam}, whether to treat the viscous term implicitly or explicitly, adding grad-div stabilization  \cite{GRVZ23, TCEK23}, and more.  The particular Uzawa form we consider comes from combining the linearization of convection from \cite{CHS15} and the grad-div stabilization from \cite{GRVZ23, TCEK23}, which yields the following algorithm.  
%One could call it a modified grad-div stabilized Uzawa, but we will refer to it simply as Uzawa.

\begin{alg}[Uzawa]
Let $\nu>0$, $f\in H^{-1}(\Omega)$, $u_0\in X$, and $p_0\in Q$ be given, and set the grad-div parameter $\gamma> 0$.  Then, the Uzawa iteration is defined by: given $(u_k,p_k)\in (X,Q)$, find $(u_{k+1},p_{k+1})\in (X,Q)$ satisfying for all $(v,q)\in (X,Q)$,
\begin{align}
 \nu(\nabla u_{k+1},\nabla v) + (u_k \cdot\nabla u_{k+1},v) - (p_k,\nabla \cdot v) + \gamma(\nabla \cdot u_{k+1},\nabla \cdot v) & = \langle f,v \rangle, \label{U1} \\
\gamma (\nabla \cdot u_{k+1},q) + (p_{k+1} - p_k,q)&=0. \label{U2}
\end{align}
\end{alg}

Because $\nabla \cdot X \subseteq Q$ (and in the discrete case, we assume $\nabla \cdot X_h\subseteq Q_h$), we can insert \eqref{U2} at step $k$ into step $k+1$ \eqref{U1} by taking $q=\nabla \cdot v$ to obtain
\[
 \nu(\nabla u_{k+1},\nabla v) + (u_k \cdot\nabla u_{k+1},v) - (p_{k-1},\nabla \cdot v) + \gamma(\nabla \cdot u_{k+1},\nabla \cdot v)  = \langle f,v \rangle + \gamma (\nabla \cdot u_k,\nabla \cdot v). 
 \]
 Repeating this pressure substitution for $k-1$, $k-2$,..., provides us with
\[
 \nu(\nabla u_{k+1},\nabla v) + (u_k \cdot\nabla u_{k+1},v) - (p_0,\nabla \cdot v) + \gamma(\nabla \cdot u_{k+1},\nabla \cdot v)  = \langle f,v \rangle + \gamma \sum_{j=1}^k (\nabla \cdot  u_j,\nabla \cdot v). 
 \]
 If the initial guess is $p_0=0$ and a divergence-free velocity, then we recover
\[
 \nu(\nabla u_{k+1},\nabla v) + (u_k \cdot\nabla u_{k+1},v) + \gamma(\nabla \cdot u_{k+1},\nabla \cdot v)  = \langle f,v \rangle + \gamma \sum_{j=0}^k (\nabla \cdot  u_j,\nabla \cdot v), 
 \]
 which is the iterated penalty method for the NSE from  \cite{C93,G89b}.  This is an interesting connection between classical methods.

\subsection{CDA Preliminaries}

We denote by $\tau_H(\Omega)$ a coarse mesh of $\Omega$ to represent partial solution data.  We require that the $\tau_H$ nodes are also nodes of $\tau_h$.  We denote $I_H$ to be an interpolant on $\Omega$  satisfying: there exists a constant $C_I$ independent of $H$ such that 
\begin{align}
    \label{prelimeq4} ||I_Hv - v|| &\leq C_IH||\nabla v|| \quad \forall v \in X \\
    \label{prelimeq5} ||I_Hv|| &\leq C_I ||v|| \quad \forall v \in X.
\end{align}
Examples of such $I_H$ are the Scott-Zhang interpolant \cite{BS08,SZ90} and the $L^2$ projection onto $P_0(\tau_H)$.  Our computations will use $I_H = P_0(\tau_H)$ but with a particular quadrature approximation on each coarse mesh element that leads to a simple CDA implementation called algebraic nudging \cite{RZ21}.

\subsection{CDA-Uzawa}

We have now provided sufficient notation and preliminaries to formally define the CDA-Uzawa scheme.

\begin{alg}[CDA-Uzawa]
Let $\nu>0$, $f\in H^{-1}(\Omega)$, and $I_H(u)$ be known with $u$ representing a corresponding NSE velocity solution.  Further,  assume $u_0\in X$ and $p_0\in Q$ are given, and set the grad-div parameter $\gamma> 0$ and the nudging parameter $\mu>0$.  Then, the CDA-Uzawa iteration is defined by: given $(u_k,p_k)\in (X,Q)$, find $(u_{k+1},p_{k+1})\in (X,Q)$ satisfying for all $(v,q)\in (X,Q)$,

\begin{align}
\mu(I_H (u_{k+1}-u),I_Hv) + \nu(\nabla u_{k+1},\nabla v) + (u_k \cdot\nabla u_{k+1},v) - (p_k,\nabla \cdot v) + \gamma(\nabla \cdot u_{k+1},\nabla \cdot v) & = \langle f,v \rangle, \label{CDAU1} \\
\gamma (\nabla \cdot u_{k+1},q) + (p_{k+1} - p_k,q)&=0. \label{CDAU2}
\end{align}
\end{alg}

Since the term $\mu(I_H u_{k+1},I_Hv)$ contributes a nonnegative and symmetric component to the linear system, the well-posedness of \eqref{CDAU1}-\eqref{CDAU2} at each iteration follows from Lax Milgram just as it does in the $\mu=0$ case \cite{C93}.

\section{Analysis of CDA-Uzawa}

In this section, we prove a convergence theorem for CDA-Uzawa to the NSE solution from which that partial solution data $I_H(u)$ was obtained.  This result quantifies how the inclusion of data using CDA improves the convergence properties of the proposed nonlinear solver.

We define the following constants, as the convergence proof gets somewhat technical.
\begin{align*}
C_0 & := 4\beta^{-2} \left( 2 (\mu H)^2 \nu^{-1} C_I^6 C_P^2   + 5 \nu   + 4 \nu\alpha^2 C_P \right), \\
C_1 & := 64\beta^{-2}M^2 C_P \nu^{-2}, \\
C_2 & := 5\beta^{-2} C_I  \left(  (  \mu^2 C_I^4 C_P^2 \lambda^{-1} + \nu + \nu\alpha^2 C_P)   5 \sqrt{2}  + \sqrt{2} \nu    \right) \alpha^2,\\
C_3 & := 5 \beta^{-2} (  \mu^2 C_I^4 C_P^2 \lambda^{-1} + \nu + \nu\alpha^2 C_P) C_0,\\
C_4 & :=5 \sqrt{2} C_I \alpha^2,\\
\lambda & := \max \{ \mu, \frac{\nu}{2 C_I^2 H^2} \}, \\
\kappa & := \max \{ H(\gamma^{-1} C_2+C_4),\gamma^{-1} (C_0+\gamma^{-1}C_3) \}.
\end{align*}

We will now prove that the Lipschitz constant associated with the CDA-Uzawa nonlinear solver is (bounded by) $\kappa$.  The convergence result of Theorem \ref{thm1} makes no explicit assumption on $\alpha$; generally, nonlinear solvers for NSE schemes require $\alpha<1$ or smaller to guarantee convergence \cite{BGHRR25}.  CDA allows for this restriction to be removed by choosing $H$ sufficiently small (i.e. by having enough partial solution data) and $\gamma$ sufficiently large, which leads to a contraction ratio of $\sim H^{1/2}\alpha$ (noting that $\alpha^2$ is a factor of both $C_2$ and $C_4$).  Thus, CDA accelerates convergence and moreover enables convergence when the problem parameters are too large for Uzawa to converge.  Our numerical tests show that CDA-Uzawa indeed converges for much higher Reynolds numbers than Uzawa.  In computations, large $\gamma$ leads to poor conditioning since the grad-div matrix is singular, and thus taking $\gamma$ too large can lead to issues with linear solver convergence and accuracy.  In our numerical tests,  $\gamma=1$ was generally sufficiently large, but in some cases $\gamma=10$ was better, suggesting the theorem's sufficient conditions may be improvable.

\begin{theorem}\label{thm1}
Let $(u,p)\in (X,Q)$ solve the NSE \eqref{eq4}-\eqref{eq5}, and suppose that $I_H(u)$ is known.  Let $(u_{k+1},p_{k+1})\in (X,Q)$ be the $k+1^{st}$ iterate of the CDA-Uzawa algorithm \eqref{CDAU1}-\eqref{CDAU2}.  Suppose $H$ is chosen small enough and $\gamma$ large enough so that $\kappa<1$, and 
$\mu \ge \frac{\nu}{2C_I^2H^2}$.  Then, CDA-Uzawa converges linearly to the NSE solution from which $I_H(u)$ is obtained:
\begin{multline}
\lambda \| u - u_{k+1} \|^2 + \frac{\nu}{4} \| \nabla (u - u_{k+1}) \|^2 + \gamma^{-1}\| p - p_{k+1} \|^2 
\\ 
\le \kappa \left(  \lambda \| u - u_{k} \|^2 + \frac{\nu}{4} \| \nabla (u - u_{k}) \|^2 + \gamma^{-1} \| p - p_{k} \|^2 \right).
\end{multline}
\end{theorem}

\begin{proof}
We begin by subtracting CDA-Uzawa \eqref{CDAU1}-\eqref{CDAU2} from the NSE \eqref{eq4}-\eqref{eq5} and setting $e_k:= u- u_k$ and $\delta_k = p - p_k$.   The error equations then become, for any $v\in X$ or $q\in Q$, 
\begin{align}
\mu(I_H e_{k+1},I_Hv) + \nu(\nabla e_{k+1},\nabla v) + (e_k \cdot\nabla u,v) + (u_k \cdot\nabla e_{k+1},v) - (\delta_k,\nabla \cdot v) + \gamma(\nabla \cdot e_{k+1},\nabla \cdot v) & = 0, \label{err1a} \\
\gamma (\nabla \cdot e_{k+1},q) - (p_{k+1} - p_k,q)&=0. \label{err1b}
\end{align}
Since $\nabla \cdot X\subseteq Q$, equation \eqref{err1b} implies that $\gamma \nabla \cdot e_{k+1} \equiv p_{k+1} - p_k \equiv \delta_{k}-\delta_{k+1}$ and thus that 
\[
(\delta_k - \gamma \nabla \cdot e_{k+1},\nabla \cdot v)=(\delta_{k+1}, \nabla \cdot v).
\]
Using this and taking $v=e_{k+1}$ in \eqref{err1a} provides us with
\begin{equation}
\mu \| I_H e_{k+1} \|^2 + \nu \| \nabla e_{k+1} \|^2 = -(e_k\cdot \nabla u,e_{k+1}) + (\delta_{k+1},\nabla \cdot e_{k+1}), \label{err2}
\end{equation}
since $(u_k \cdot \nabla e_{k+1},e_{k+1})=0$.  

Recall that it is proven in (\cite{C93}, equation (3.12)) that 
\begin{equation}
(\delta_{k+1},p_{k+1}-p_k) \le (\delta_{k+1},\delta_k). \label{codina}
\end{equation}
Next, we bound the first right hand side term of \eqref{err2} using \eqref{bbound} along with \eqref{prelimeq3} and the definition of $\alpha$, and we use \eqref{codina} and Cauchy-Schwarz on the second term to get the bound
\begin{align}
\mu \| I_H e_{k+1} \|^2 + \nu \| \nabla e_{k+1} \|^2 
& \le M \| e_k \|^{1/2} \| \nabla e_k \|^{1/2} \| \nabla u \| \| \nabla e_{k+1} \| +    ( \delta_{k+1},\nabla \cdot e_{k+1})  \nonumber \\
& \le \nu \alpha \| e_k \|^{1/2} \| \nabla e_k \|^{1/2}  \| \nabla e_{k+1} \| +  \gamma^{-1}  ( \delta_{k+1},p_{k+1}-p_k)  \nonumber \\
& \le M \| e_k \|^{1/2} \| \nabla e_k \|^{1/2} \| \nabla u \| \| \nabla e_{k+1} \| + \gamma^{-1} \| \delta_{k+1} \| \| \delta_{k} \| \nonumber \\
& \le \frac{\nu}{4} \| \nabla e_{k+1} \|^2 + \nu\alpha^2 \| e_k \| \| \nabla e_k \| + \gamma^{-1} \| \delta_{k+1} \| \| \delta_{k} \|, \label{err3}
\end{align}
with the last step coming from Young's inequality.

We lower bound the left hand side following \cite{GN20} using \eqref{prelimeq4} and the triangle inequality to get that
\[
\lambda \| e_{k+1} \|^2 + \frac{\nu}{2} \| \nabla e_{k+1} \|^2 \le \mu \| I_H e_{k+1} \|^2 + \nu \| \nabla e_{k+1} \|^2,
\]
where $\lambda =  \min\{ \mu,\frac{\nu}{2C_I^2 H^2} \}$.  Combining this with \eqref{err3} and then bounding the second right hand side term in \eqref{err3} with Young's inequality, we obtain
\begin{align}
\lambda \| e_{k+1} \|^2 + \frac{\nu}{4} \| \nabla e_{k+1} \|^2 & \le \nu\alpha^2 \| e_k \| \| \nabla e_k \|\frac{(\sqrt{2}C_I H)^{1/2}}{(\sqrt{2}C_I H)^{1/2}}  + \gamma^{-1} \| \delta_{k+1} \| \| \delta_{k} \| \nonumber \\
& \le 
\alpha^2 \left( \frac{\nu}{4} \sqrt{2}C_I H \| \nabla e_k \|^2 + \frac{\nu}{\sqrt{2}C_I H} \| e_k \|^2 \right)
+ \gamma^{-1} \| \delta_{k+1} \| \| \delta_{k} \| \nonumber \\
& = 
\sqrt{2}C_I H \alpha^2 \left(  \frac{\nu}{4}\| \nabla e_k \|^2 + \lambda \| e_k \|^2 \right)+ \gamma^{-1} \| \delta_{k+1} \| \| \delta_{k} \|, \label{err4}
\end{align}
with the last step assuming that $\mu\ge \frac{\nu}{2C_I^2H^2}$.\\
\indent We now desire to obtain a bound on $\delta_{k+1}$. We do this first by adding and subtracting $u$ to $u_k$ in $(u_k \cdot \nabla e_{k+1}, v)$ in \eqref{err1a}. Then, using the inf-sup condition, we get from \eqref{err1a}-\eqref{err1b} that
\begin{multline}
\beta \| \delta_{k+1} \| \le \mu C_I^2 C_P \| e_{k+1} \| + \nu \| \nabla e_{k+1} \|  \\ + \nu\alpha \| e_k \|^{1/2} \| \nabla e_k \|^{1/2} + M C_P^{1/2} \| \nabla e_k \| \| \nabla e_{k+1} \| + \nu\alpha C_P^{1/2} \| \nabla e_{k+1} \|, \label{err4b}
\end{multline}
thanks to Poincar\'e, \eqref{prelimeq5}, Cauchy-Schwarz, \eqref{bbound}, and \eqref{prelimeq3}.  Using this upper bound in \eqref{err4} provides the estimate
\begin{multline}
\lambda \| e_{k+1} \|^2 + \frac{\nu}{4} \| \nabla e_{k+1} \|^2
\le \sqrt{2}C_I H \alpha^2 \left(  \frac{\nu}{4}\| \nabla e_k \|^2 + \lambda \| e_k \|^2 \right) 
+ \gamma^{-1}\beta^{-1} \mu C_I^2 C_P \| \delta_{k} \|  \| e_{k+1} \| \\
+ \gamma^{-1}\beta^{-1}  \nu  \| \delta_{k} \| \| \nabla e_{k+1} \|
+ \gamma^{-1}\beta^{-1}  \nu\alpha  \| \delta_{k} \| \| e_k \|^{1/2} \| \nabla e_k \|^{1/2}
+ \gamma^{-1}\beta^{-1} M C_P^{1/2}  \| \delta_{k} \| \| \nabla e_k \| \| \nabla e_{k+1} \| \\
+ \gamma^{-1}\beta^{-1}\nu\alpha C_P^{1/2}  \| \delta_{k} \|\| \nabla e_{k+1} \|.\label{err5}
\end{multline}
We now bound each term in \eqref{err5} with a $\delta_k$.  For the first two and the last two terms, we use Young's inequality and obtain
\begin{align*}
 \gamma^{-1}\beta^{-1} \mu C_I^2 C_P \| \delta_{k} \|  \| e_{k+1} \| & \le  \gamma^{-2}\beta^{-2}\lambda^{-1} \mu^2 C_I^4 C_P^2 \| \delta_{k} \|^2 + \frac{\lambda}{4} \| e_{k+1} \|^2, \\
 \gamma^{-1}\beta^{-1}  \nu  \| \delta_{k} \| \| \nabla e_{k+1} \| & \le 4\gamma^{-2}\beta^{-2}  \nu  \| \delta_{k} \|^2 + \frac{\nu}{16} \| \nabla e_{k+1} \|^2, \\
  \gamma^{-1}\beta^{-1} M C_P^{1/2}  \| \delta_{k} \| \| \nabla e_k \| \| \nabla e_{k+1} \|
 & \le  4\nu^{-1} \gamma^{-2}\beta^{-2} M^2 C_P  \| \delta_{k} \|^2 \| \nabla e_k \|^2 + \frac{\nu}{16} \| \nabla e_{k+1} \|^2, \\
  \gamma^{-1}\beta^{-1}\nu\alpha C_P^{1/2}  \| \delta_{k} \|\| \nabla e_{k+1} \| & \le  4\gamma^{-2}\beta^{-2}\nu\alpha^2 C_P  \| \delta_{k} \|^2 + \frac{\nu}{16} \| \nabla e_{k+1} \|^2. 
\end{align*}
For the third term in \eqref{err5} with a $\delta_k$, we apply Young's inequality twice with the second time following \eqref{err4} to get the bound
\begin{align*}
\gamma^{-1}\beta^{-1}  \nu\alpha  \| \delta_{k} \| \| e_k \|^{1/2} \| \nabla e_k \|^{1/2}
& \le
\gamma^{-2 }\beta^{-2}  \nu \| \delta_{k} \|^2 +  \frac14 \nu\alpha^2 \| e_k \| \| \nabla e_k \| \\
& \le
\gamma^{-2 }\beta^{-2}  \nu  \| \delta_{k} \|^2 +  \frac14 \sqrt{2} C_I H\alpha^2 \left( \frac{\nu}{4} \| \nabla e_k \|^2 + \lambda \| e_k \|^2 \right).
\end{align*}

Combining these bounds, using the assumption on $\mu$ which implies $\lambda=\frac{\nu}{2C_I^2 H^2}$ with \eqref{err5}, and reducing yields
\begin{multline}
\lambda \| e_{k+1} \|^2 + \frac{\nu}{4} \| \nabla e_{k+1} \|^2
\le 
  5 \sqrt{2} C_I H\alpha^2 \left( \frac{\nu}{4} \| \nabla e_k \|^2 + \lambda \| e_k \|^2 \right) \\
 + 4\gamma^{-2}\beta^{-2} \left( 2 (\mu H)^2 \nu^{-1} C_I^6 C_P^2   + 5 \nu   + 4 \nu\alpha^2 C_P \right)
\| \delta_{k} \|^2 + 16 \gamma^{-2}\beta^{-2} M^2 C_P \nu^{-1} \| \nabla e_k \|^2 \| \delta_{k} \|^2 .
\label{err6}
\end{multline}
Next, for notational simplicity, we denote
\begin{align*}
U_{k} &:= \lambda \| e_k \|^2 + \frac{\nu}{4} \| \nabla e_{k} \|^2,\\
 P_k &:= \| \delta_k \|^2.
\end{align*}
The estimate \eqref{err6} can now be written as
\begin{equation}
U_{k+1} \le 5 \sqrt{2} C_I H\alpha^2 U_k + \gamma^{-2} C_0 P_k + C_1 \gamma^{-2} U_k P_k. \label{err7}
\end{equation}
Next, squaring both sides of \eqref{err4b} and reducing provides
\begin{equation*}
P_{k+1} \le 5\beta^{-2} \bigg(  \mu^2 C_I^4 C_P^2 \lambda^{-1} U_{k+1}  + \nu U_{k+1}    + \nu^2\alpha^2 \| e_k \| \|\nabla e_k \| + M^2 C_P \nu^{-2} U_{k} U_{k+1}  + \nu\alpha^2 C_P U_{k+1} \bigg), \label{err8}
\end{equation*}
and further reduction with Young's inequality on the higher order term leads to
\begin{multline}
P_{k+1} \le 5\beta^{-2} (  \mu^2 C_I^4 C_P^2 \lambda^{-1} + \nu + \nu\alpha^2 C_P) U_{k+1}   \\  + 5\beta^{-2}   \nu^{3/2}\alpha^2 \lambda^{-1/2} U_k + \frac52 \beta^{-2}   M^2 C_P \nu^{-2} U_{k+1}^2 + \frac52 \beta^{-2}   M^2 C_P \nu^{-2} U_{k}^2. \label{err91}
\end{multline}
Now, using \eqref{err7} in \eqref{err91} and substituting in for $\lambda^{-1/2}$ gives us
\begin{multline}
P_{k+1} \le 
5\beta^{-2} C_I \alpha^2 H \left(  (  \mu^2 C_I^4 C_P^2 \lambda^{-1} + \nu + \nu\alpha^2 C_P)   5 \sqrt{2}  + \sqrt{2} \nu    \right) U_k \\
+ 5 \gamma^{-2} \beta^{-2} (  \mu^2 C_I^4 C_P^2 \lambda^{-1} + \nu + \nu\alpha^2 C_P) C_0 P_k 
+ 5\gamma^{-2} \beta^{-2} (  \mu^2 C_I^4 C_P^2 \lambda^{-1} + \nu + \nu\alpha^2 C_P) C_1 U_k P_k \\
+ \frac52 \beta^{-2}   M^2 C_P \nu^{-2} \left( 5 \sqrt{2} C_I H\alpha^2 U_k + \gamma^{-2} C_0 P_k + C_1 \gamma^{-2} U_k P_k\right)^2 + \frac52 \beta^{-2}   M^2 C_P \nu^{-2} U_{k}^2. \label{err9}
\end{multline}
Thus, after dropping higher order terms, we obtain
\begin{align*}
U_{k+1} & \le  H C_4 U_k + \gamma^{-2} C_0 P_k, \\
P_{k+1} & \le 
H C_2 U_k 
+  \gamma^{-2} C_3 P_k. 
\end{align*}
Adding the inequalities after scaling the second one by $\gamma^{-1}$ yields
\begin{align*}
U_{k+1} + \gamma^{-1} P_{k+1} & \le H(\gamma^{-1}C_2+C_4)U_k + \gamma^{-2} (C_0+\gamma^{-1}C_3)P_k \\
& = H(\gamma^{-1}C_2+C_4)U_k + \gamma^{-1} (C_0+\gamma^{-1}C_3) \gamma^{-1} P_k \\
& \le \max \{ H(\gamma^{-1} C_2+C_4),\gamma^{-1} (C_0+\gamma^{-1}C_3) \} \left( U_k + \gamma^{-1} P_k\right) \\
& = \kappa (U_k + \gamma^{-1} P_k), 
\end{align*}
which implies a contraction under the assumption that $\kappa<1$.

\end{proof}

\subsection{The case of noisy data}

We now consider convergence of CDA-Uzawa when the partial solution data is noisy.  We observe similar linear convergence as in the case of no noise above but only down to the size of the noise. We again define the following constants for notational simplicity.  Note that these constants change only slightly from those used (without tildes) above in Theorem \ref{thm1}. 
\begin{align*}
    \tilde{C_0} & := 4\beta^{-2} \left( 2 (\mu H)^2 \nu^{-1} C_I^6 C_P^2   + 5 \nu   + 4 \nu\alpha^2 C_P + \mu^2 C_I^4 C_P^2 \right), \\
 \tilde{C_1} & := 64 \beta^{-2}M^2 C_P \nu^{-2},\\
 \tilde{C_2} & := 6\beta^{-2} C_I  \left(  (  \mu^2 C_I^4 C_P^2 \hat{\lambda}^{-1} + \nu + \nu\alpha^2 C_P)   5 \sqrt{2}  + \nu    \right) \alpha^2,\\
\tilde{C_3} & := 6 \beta^{-2} (  \mu^2 C_I^4 C_P^2 \hat{\lambda}^{-1} + \nu + \nu\alpha^2 C_P) C_0,\\
\tilde{C_4} & :=5 \sqrt{2} C_I \alpha^2, \\
\tilde{C_5} & := 6\beta^{-2}\mu (  3\mu^2 C_I^4 C_P^2 \hat{\lambda}^{-1} + 3\nu + 3\nu\alpha^2 C_P + \mu C_I^2C_P^2), \\
\tilde \kappa & := \max \{ H(\gamma^{-1} \tilde{C_2}+\tilde{C_4}),\gamma^{-1} (\tilde{C_0}+\gamma^{-1}\tilde{C_3}) \}, \\
\hat \lambda & :=  \min\{ \frac{\mu}{2},\frac{\nu}{2C_I^2 H^2} \}.
\end{align*}

%
%\begin{remark}
%    Note that these constants change only slightly from Theorem \ref{thm1} with the most notable change being that we now have $\hat{\lambda}$, which differs slightly from the original definition of $\lambda$. Despite this change, the dependency on the parameters for these constants is the same from Theorem \ref{thm1}.
%\end{remark}

\noindent We consider $\epsilon$ to be the noise in the data, and thus $\epsilon(x_i)$ i=1,.2,...,N so that $I_H \epsilon$ is known.  

\begin{theorem} \label{thm2}
Let $(u,p)\in (X,Q)$ solve the NSE \eqref{eq4}-\eqref{eq5} and suppose that $I_H(u+\epsilon)$ is known. Let $(u_{k+1},p_{k+1})\in (X,Q)$ be the $k+1^{st}$ iterate of the CDA-Uzawa algorithm \eqref{CDAU1}-\eqref{CDAU2}.  Suppose $H$ is chosen small enough and $\gamma$ is chosen large enough so that $\tilde \kappa<1$, and
$\mu \ge \frac{\nu}{C_I^2H^2}$.
Then, CDA-Uzawa error satisfies 
\begin{multline*}
    \hat{\lambda} \| u - u_{k+1} \|^2 + \frac{\nu}{4} \| \nabla (u - u_{k+1}) \|^2 + \gamma^{-1}\| p - p_{k+1} \|^2 \\
\le \tilde \kappa \left(  \hat{\lambda} \| u - u_{k} \|^2 + \frac{\nu}{4} \| \nabla (u - u_{k}) \|^2  + \gamma^{-1} \| p - p_{k} \|^2 \right) + \left( 1 + \tilde{C_5} \gamma^{-1} \right) \| I_H \epsilon \|^2.
\end{multline*}

\end{theorem}

\begin{remark}
We observe from the theorem that in the case of noisy partial solution data, the Lipschitz constant of the CDA-Uzawa solver is $\sim H^{1/2} \alpha$, just as in the case of no noise.  However, convergence is only down to a level dependent on the size of the noise.
\end{remark}

\begin{proof}
This proof follows similar steps as the proof of Theorem \ref{thm1}. The only difference is that the nudging term is now $\mu (I_H (u + \epsilon - u_{k+1}),I_H v)$ and is, therefore, handled differently.

We begin by subtracting \eqref{CDAU1}-\eqref{CDAU2} with the altered nudging term from \eqref{eq4}-\eqref{eq5} and setting $e_k:= u- u_k$ and $\delta_k = p - p_k$.   The error equations then become, for any $v\in X$ or $q\in Q$, 
\begin{align}
\mu(I_H e_{k+1},I_Hv) + \nu(\nabla e_{k+1},\nabla v) + (e_k \cdot\nabla u,v) + (u_k \cdot\nabla e_{k+1},v) & \nonumber \\ - (\delta_k,\nabla \cdot v) + \gamma(\nabla \cdot e_{k+1},\nabla \cdot v) & = -\mu (I_H \epsilon,I_H v), \label{err1aN} \\
\gamma (\nabla \cdot e_{k+1},q) - (p_{k+1} - p_k,q)&=0. \label{err1bN}
\end{align}
Following  as in the proof of Theorem \ref{thm1}  by using that $\gamma \nabla \cdot e_{k+1} \equiv \delta_{k}-\delta_{k+1}$ and choosing $v=e_{k+1}$ in \eqref{err1aN}, we obtain
\begin{equation}
\mu \| I_H e_{k+1} \|^2 + \nu \| \nabla e_{k+1} \|^2 = -(e_k\cdot \nabla u,e_{k+1}) + (\delta_{k+1},\nabla \cdot e_{k+1}) -\mu (I_H \epsilon,I_H e_{k+1}). \label{err2N}
\end{equation}
We bound the last term in \eqref{err2N} using Cauchy-Schwarz and Young's inequalities to get
\[
\mu (I_H \epsilon,I_H e_{k+1}) \le \frac{\mu}{2} \| I_H \epsilon \|^2 + \frac{\mu}{2} \| I_H e_{k+1} \|^2.
\]
Combining this with estimates \eqref{err3}-\eqref{err4} from the proof of Theorem \ref{thm1} for the rest of the terms in \eqref{err2N}, we get that
\begin{align}
\hat \lambda \| e_{k+1} \|^2 + \frac{\nu}{4} \| \nabla e_{k+1} \|^2 
%& \le \nu\alpha^2 \| e_k \| \| \nabla e_k \|\frac{(\sqrt{2}C_I H)^{1/2}}{(\sqrt{2}C_I H)^{1/2}}  + \gamma^{-1} \| \delta_{k+1} \| \| \delta_{k} \| \nonumber \\
& \le 
\sqrt{2}C_I H \alpha^2 \left(  \frac{\nu}{4}\| \nabla e_k \|^2 + \hat{\lambda} \| e_k \|^2 \right)+ \gamma^{-1} \| \delta_{k+1} \| \| \delta_{k} \| + \frac{\mu}{2} \| I_H \epsilon \|^2, \label{err4N}
\end{align}
under the assumption $\mu\ge \frac{\nu}{C_I^2H^2}$.

%where $\hat \lambda =  \min\{ \frac{\mu}{2},\frac{\nu}{2C_I^2 H^2} \}$ and assuming $\mu\ge \frac{\nu}{C_I^2H^2}$.

Using the inf-sup condition, we follow the analysis above in \eqref{err4b} to get from \eqref{err1aN}-\eqref{err1bN} to
\begin{multline}
\beta \| \delta_{k+1} \| \le \mu C_I^2 C_P \| e_{k+1} \| + \nu \| \nabla e_{k+1} \|  \\ + \nu\alpha \| e_k \|^{1/2} \| \nabla e_k \|^{1/2} + M C_P^{1/2} \| \nabla e_k \| \| \nabla e_{k+1} \| + \nu\alpha C_P^{1/2} \| \nabla e_{k+1} \| + \mu C_I C_P \| I_H \epsilon \|, \label{err4bN}
\end{multline}
noting that the only difference from \eqref{err4b} is the nudging term that arises from the right side of \eqref{err1aN}.  Using this estimate in \eqref{err4N} yields
\begin{multline}
\hat{\lambda} \| e_{k+1} \|^2 + \frac{\nu}{4} \| \nabla e_{k+1} \|^2
\le \sqrt{2}C_I H \alpha^2 \left(  \frac{\nu}{4}\| \nabla e_k \|^2 + \hat{\lambda} \| e_k \|^2 \right) 
+ \gamma^{-1}\beta^{-1} \mu C_I^2 C_P \| \delta_{k} \|  \| e_{k+1} \| \\
+ \gamma^{-1}\beta^{-1}  \nu  \| \delta_{k} \| \| \nabla e_{k+1} \|
+ \gamma^{-1}\beta^{-1}  \nu\alpha  \| \delta_{k} \| \| e_k \|^{1/2} \| \nabla e_k \|^{1/2}
+ \gamma^{-1}\beta^{-1} M C_P^{1/2}  \| \delta_{k} \| \| \nabla e_k \| \| \nabla e_{k+1} \| \\
+ \gamma^{-1}\beta^{-1}\nu\alpha C_P^{1/2}  \| \delta_{k} \|\| \nabla e_{k+1} \|
+ \gamma^{-1}\beta^{-1} \mu C_I C_P \| \delta_k \| \| I_H \epsilon \| 
+ \frac{\mu}{2}\| I_H \epsilon \|^2.\label{err5N}
\end{multline}
Each term in \eqref{err5N} on the right hand side gets bounded as in the proof of Theorem \ref{thm1} (with $\lambda$ replaced by $\hat{\lambda}$) except the penultimate term.  For this term, we get that
\[
\gamma^{-1}\beta^{-1} \mu C_I C_P \| \delta_k \| \| I_H \epsilon \|
\le 
\gamma^{-2}\beta^{-2} \mu C_I^2 C_P^2 \| \delta_k \|^2 +  \frac{\mu}{4} \| I_H \epsilon \|^2.
\]
Combining this bound with \eqref{err6}, which bounds the rest of the terms from \eqref{err5N}, provides the estimate 
\begin{multline}
\hat{\lambda} \| e_{k+1} \|^2 + \frac{\nu}{4} \| \nabla e_{k+1} \|^2
\le 
  5 \sqrt{2} C_I H\alpha^2 \left( \frac{\nu}{4} \| \nabla e_k \|^2 + \hat{\lambda} \| e_k \|^2 \right) +  3\mu \| I_H \epsilon \|^2 \\
 + 4\gamma^{-2}\beta^{-2} \left( 2 (\mu H)^2 \nu^{-1} C_I^6 C_P^2   + 5 \nu   + 4 \nu\alpha^2 C_P + \mu C_I^2 C_P^2 \right)
\| \delta_{k} \|^2 \\+ 16 \gamma^{-2} \beta^{-2} M^2 C_P \nu^{-1} \| \nabla e_k \|^2 \| \delta_{k} \|^2,
\label{err6N}
\end{multline}
noting that the assumption on $\mu$ implies $\hat{\lambda}=\frac{\nu}{C_I^2 H^2}$.\\
Next, for notational simplicity, we denote
\begin{align*}
U_{k} &:= \hat{\lambda} \| e_k \|^2 + \frac{\nu}{4} \| \nabla e_{k} \|^2,\\
 P_k &:= \| \delta_k \|^2.
\end{align*}
Note that $P_k$ is the same as in the proof of Theorem \ref{thm1}, and $U_k$ only differs with $\hat{\lambda}$.

The estimate \eqref{err6N} can now be written as
\begin{equation}
U_{k+1} \le 5 \sqrt{2} C_I H\alpha^2 U_k + \gamma^{-2} \tilde{C_0} P_k + \tilde{C_1} \gamma^{-2} U_k P_k + 3\mu \| I_H \epsilon \|^2. \label{err7N}
\end{equation}

Next, squaring both sides of \eqref{err4bN} and reducing provides
\begin{multline*}
P_{k+1} \le 6\beta^{-2} \bigg(  \mu^2 C_I^4 C_P^2 \hat{\lambda}^{-1} U_{k+1}  + \nu U_{k+1}    + \nu^2\alpha^2 \| e_k \| \|\nabla e_k \| + M^2 C_P \nu^{-2} U_{k} U_{k+1} \bigg. \\ \bigg. + \nu\alpha^2 C_P U_{k+1} 
+ \mu^2 C_I^2 C_P^2 \| I_H \epsilon \|^2  \bigg),
\end{multline*}
and further reduction with Young's inequality on the higher order term leads to
\begin{multline}
P_{k+1} \le 6\beta^{-2} (  \mu^2 C_I^4 C_P^2 \hat{\lambda}^{-1} + \nu + \nu\alpha^2 C_P) U_{k+1}   \\  + 6\beta^{-2}   \nu^{3/2}\alpha^2 \hat{\lambda}^{-1/2} U_k + 3 \beta^{-2}   M^2 C_P \nu^{-2} U_{k+1}^2\\ + 3 \beta^{-2}   M^2 C_P \nu^{-2} U_{k}^2  + 6\beta^{-2} \mu^2 C_I^2 C_P^2 \| I_H\epsilon \|^2. \label{err9N}
\end{multline}
Now, using \eqref{err7N} in \eqref{err9N}, substituting in for $\hat{\lambda}^{-1/2}$, and dropping higher order terms gives us
\begin{multline}
P_{k+1} \le 
6\beta^{-2} C_I \alpha^2 H \left(  (  \mu^2 C_I^4 C_P^2 \hat{\lambda}^{-1} + \nu + \nu\alpha^2 C_P)   5 \sqrt{2}  +  \nu    \right) U_k \\
+ 6 \gamma^{-2} \beta^{-2} (  \mu^2 C_I^4 C_P^2 \hat{\lambda}^{-1} + \nu + \nu\alpha^2 C_P) \tilde{C_0} P_k 
%+ 5\gamma^{-2} \beta^{-2} (  \mu^2 C_I^4 C_P^2 \lambda^{-1} + \nu + \nu\alpha^2 C_P) C_1 U_k P_k \\
%+ \frac52 \beta^{-2}   M^2 C_P \nu^{-2} \left( 5 \sqrt{2} C_I H\alpha^2 U_k + \gamma^{-2} C_0 P_k + C_1 \gamma^{-2} U_k P_k\right)^2 + \frac52 \beta^{-2}   M^2 C_P \nu^{-2} U_{k}^2. 
\\ + 6\beta^{-2}\mu (  3\mu^2 C_I^4 C_P^2 \hat{\lambda}^{-1} + 3\nu + 3\nu\alpha^2 C_P + \mu C_I^2C_P^2) \| I_H \epsilon \|^2.
\label{err9Na}
\end{multline}
This reduces to
\begin{align*}
U_{k+1} & \le  H \tilde{C_4} U_k + \gamma^{-2} \tilde{C_0} P_k + \mu \| I_H \epsilon \|^2, \\
P_{k+1} & \le 
H \tilde{C_2} U_k 
+  \gamma^{-2} \tilde{C_3} P_k  
+ \tilde{C_5} \| I_H \epsilon \|^2, 
\end{align*}
Adding the inequalities after scaling $P_k$ by $\gamma^{-1}$ yields
\begin{align*}
U_{k+1} + \gamma^{-1} P_{k+1} & \le H(\gamma^{-1}\tilde{C_2}+\tilde{C_4})U_k + \gamma^{-2} (\tilde{C_0}+\gamma^{-1}\tilde{C_3})P_k + \left( 1 + \tilde{C_5} \gamma^{-1} \right) \| I_H \epsilon \|^2 \\
& = H(\gamma^{-1}\tilde{C_2}+\tilde{C_4})U_k + \gamma^{-1} (\tilde{C_0}+\gamma^{-1}\tilde{C_3}) \gamma^{-1} P_k +  \left( 1 + \tilde{C_5} \gamma^{-1} \right) \| I_H \epsilon \|^2\\
& \le \max \{ H(\gamma^{-1} \tilde{C_2}+\tilde{C_4}),\gamma^{-1} (\tilde{C_0}+\gamma^{-1}\tilde{C_3}) \} \left( U_k + \gamma^{-1} P_k\right) +  \left( 1 + \tilde{C_5} \gamma^{-1} \right) \| I_H \epsilon \|^2\\
& \le \tilde \kappa (U_k + \gamma^{-1} P_k) + \left( 1 + \tilde{C_5} \gamma^{-1} \right) \| I_H \epsilon \|^2. 
\end{align*}
Since $\tilde \kappa<1$ by assumption, we obtain the bound
\[
U_{k+1} + \gamma^{-1} P_{k+1} \le \tilde \kappa^{k+1} \left( U_0 + \gamma^{-1} P_0 \right) + \frac{1-\tilde \kappa^{k+1}}{1-\tilde \kappa}  \left( 1 + \tilde{C_5} \gamma^{-1} \right) \| I_H \epsilon \|^2.
\]

\end{proof}

\section{Numerical Experiments}

We now give results for several numerical tests to illustrate our theory above and give comparisons to existing methods.  We use $(X_h,Q_h)=(P_k,P_{k-1}^{disc})$ Scott-Vogelius (SV) elements with $k=d$ in all tests except one where we show that $(P_2,P_1)$ Taylor-Hood (TH) elements are not appropriate for use with CDA-Uzawa; this is expected since the theory assumes $\nabla \cdot X_h \subseteq Q_h$, which holds for SV but not for TH.  All meshes are barycenter refined (Alfeld split) triangle meshes in 2D and barycenter refined tetrahedral meshes in 3D.  Note that these elements with these meshes are known to be inf-sup stable pairs \cite{arnold:qin:scott:vogelius:2D,GS19,JLMNR17,scott:vogelius:conforming, zhang:scott:vogelius:3D}.  

Convergence in our tests is measured in the norm
\[
\| (v,q) \|_{*} := \left( \| \nabla v \|^2 + \| q \|^2 \right)^{1/2}, 
\] 
which is the natural norm of the velocity and pressure spaces.  While the theory uses an additional $H$-dependent $L^2$ norm in the convergence estimate, using the $*$-norm allows for clearer comparison when varying $H$ and $\gamma$. To measure convergence and for data assimilation, the true solution is taken as a computed NSE solution found using the AA-Picard-Newton iteration \cite{PRTX25} and also continuation methods for some tests.  For the TH true solution, the NSE solution uses grad-div stabilization with the same $\gamma$ used by CDA-Uzawa.  The initial guesses are taken to be $u_0=0$ and $p_0=0$ for all tests, and $\mu=1$ in all tests.  We note that we also tested with $\mu=100$ and saw nearly indistinguishable results.

For linear solvers in CDA-Uzawa, we use a direct solver since CDA-Uzawa decouples the system, so the momentum equation is solved independently (and there is no pressure solve for Step 2 since SV elements are used; it is just a calculation).  However, after the nonlinear residual drops below $10^{-2}$ in the $*$-norm, we save the LUPQ factorization for CDA-Uzawa Step 1 every 5 iterations, and use it to precondition a GMRES solver for the following 4 nonlinear iterations.  When GMRES is used, convergence to $10^{-8}$ is found within 3-4 iterations on average.  For the test where CDA-Picard is used, we follow the strategy of \cite{HR13,benzi} by adding grad-div stabilization to the system (which has no effect on the solution since incompressibility is strongly enforced) and using an approximate block LU factorization as a preconditioner where the Schur complement is approximated by the pressure mass matrix.

Lastly, we note that none of the nonlinear solvers tested below were equipped with any additional stabilization or acceleration methods.  It is expected that all methods below could be further improved in terms of convergence rate and robustness if techniques such as Anderson acceleration \cite{PR25} were applied.

\subsection{Test problem descriptions}

We now describe the test problems to be used in the numerical experiments below.

\subsubsection{2D driven cavity}

\begin{figure}[ht]
\center
 $Re=5000$ \hspace{1in}  $Re=10000$ \\
\includegraphics[width = .3\textwidth, height=.28\textwidth,viewport=115 45 465 390, clip]{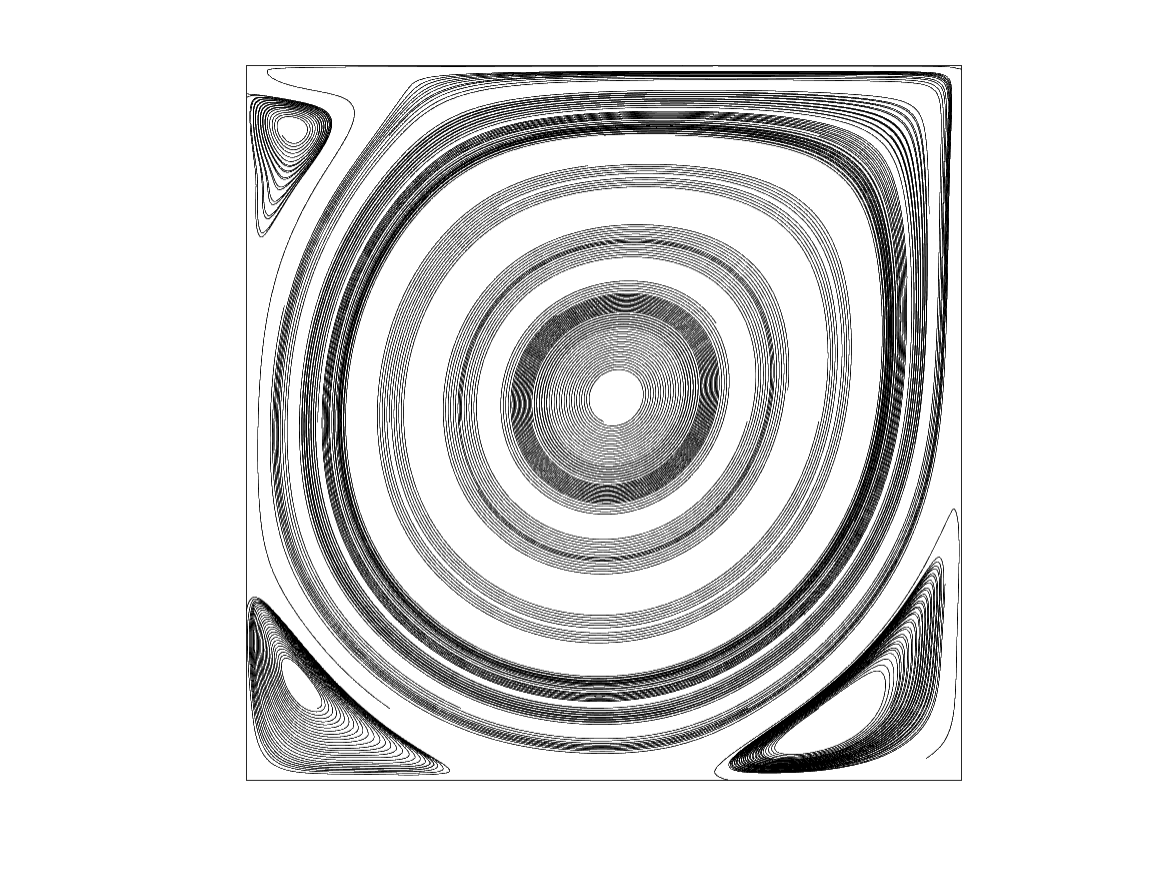}  
\includegraphics[width = .3\textwidth, height=.28\textwidth,viewport=115 45 465 390, clip]{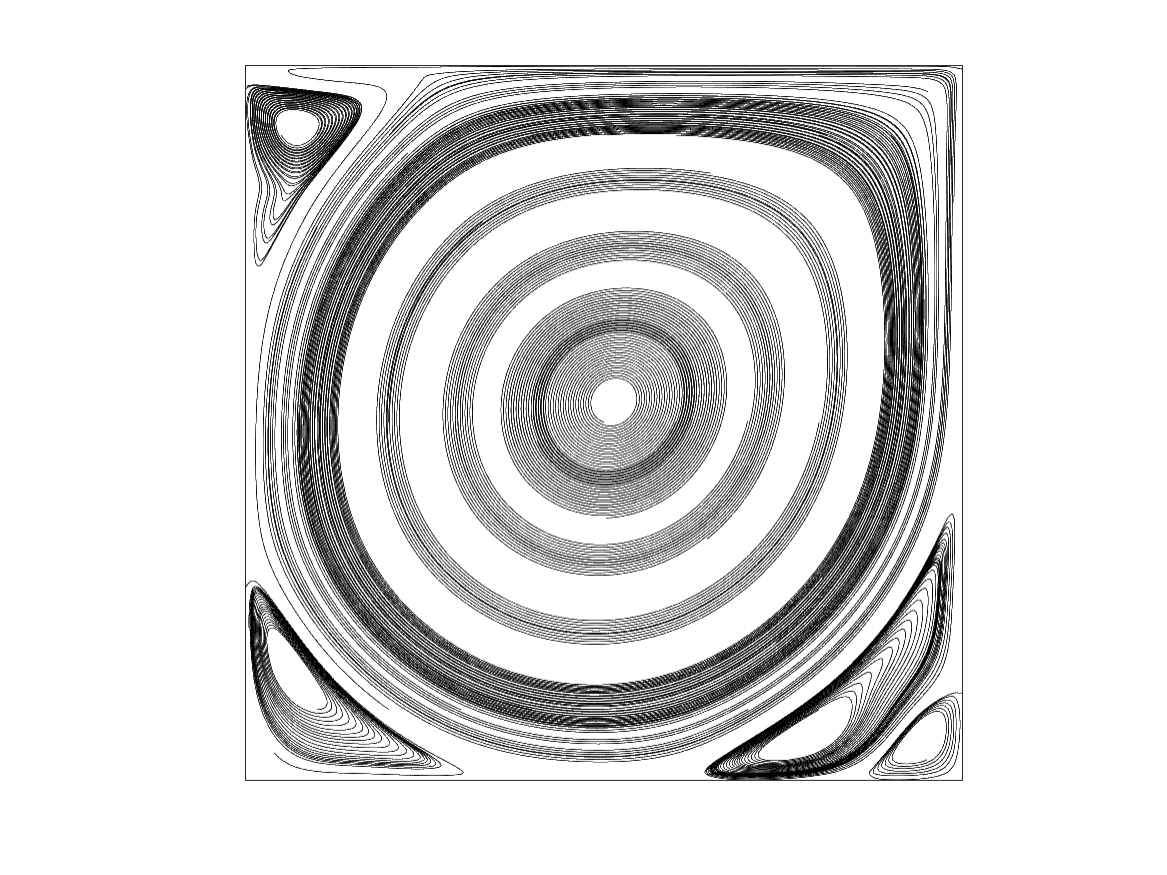}  
\caption{\label{dc2d} The plots above show streamlines of the solution of the 2D driven cavity problem solution at varying $Re$ taken from our finest mesh simulations.}
\end{figure}

The classical 2D driven cavity problem has domain $\Omega=(0,1)^2$, there is no forcing $f=0$, homogeneous Dirichlet velocity boundary conditions are enforced on the sides and bottom, and $[ 1,0 ]^T$ is enforced on the top (lid).  For this problem, $Re=\nu^{-1}$, and we compute with $Re$=5000 and 10000. Note that all of our solutions are in agreement with the literature \cite{ECG05}, see Figure \ref{dc2d}.  We compute on barycenter refined uniform $h=\frac{1}{64},\ \frac{1}{128},\ \frac{1}{196}$ meshes, which provide 99K, 394K, and 923K velocity degrees of freedom (dof) using $(P_2,P_1^{disc})$ SV elements, respectively.  Convergence results for the three meshes were always nearly identical, so results are only shown for the finest mesh.

\subsubsection{3D driven cavity}

\begin{figure}[H]
\center
$Re$=1000 \\
\includegraphics[width = .9\textwidth, height=.27\textwidth,viewport=120 0 1180 340, clip]{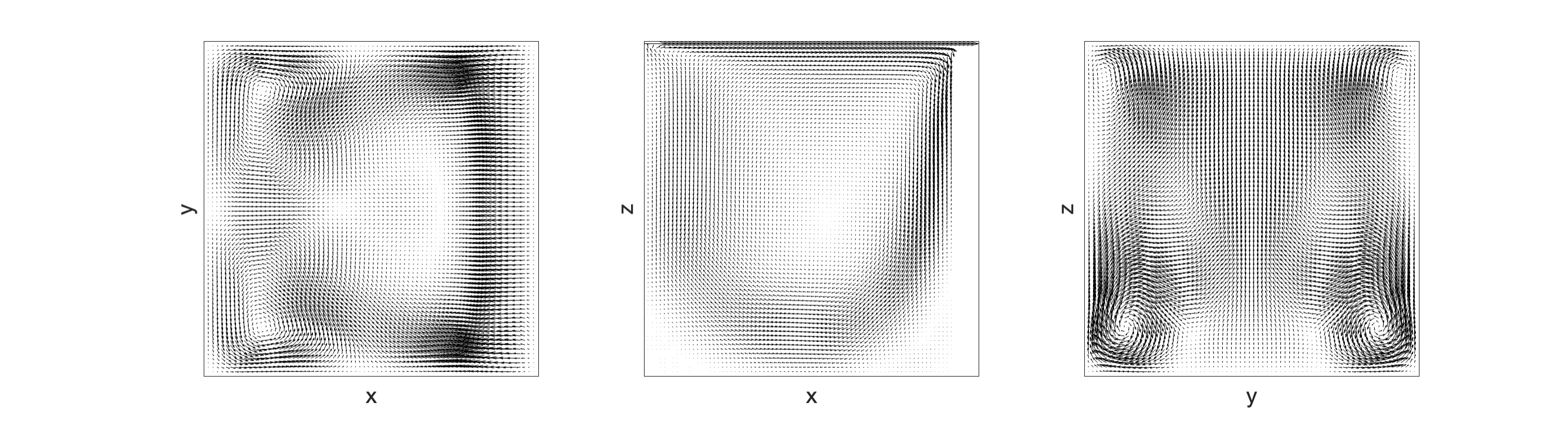}  \\
$Re$=2000 \\
\includegraphics[width = .9\textwidth, height=.27\textwidth,viewport=120 0 1180 340, clip]{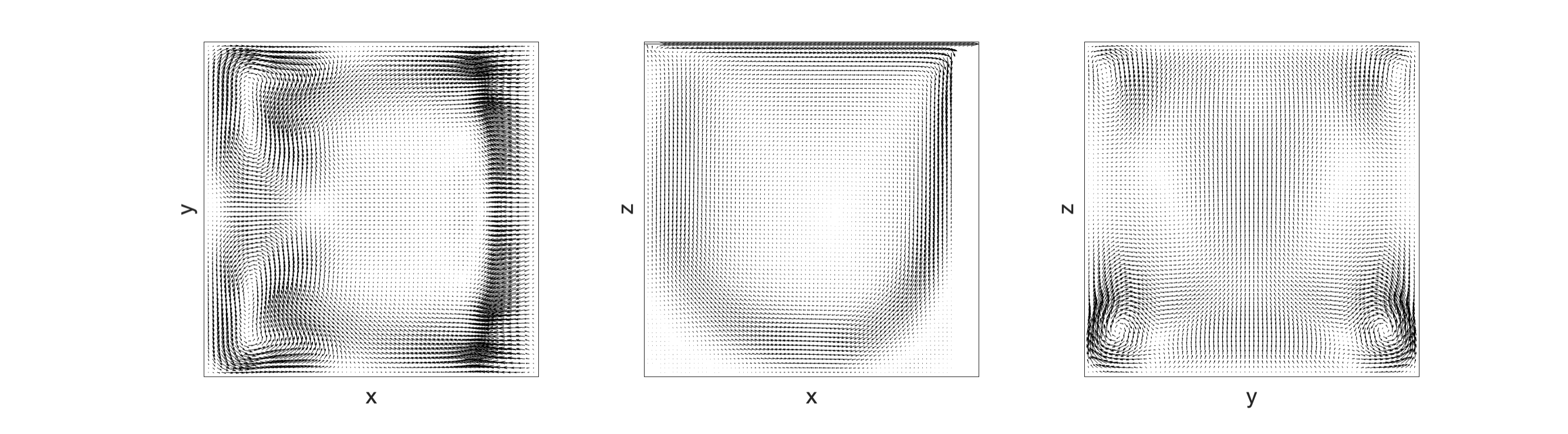}  
\caption{\label{cav3d} The plots above show $Re$=1000 (top) and $Re$=2000 (bottom) 3D driven cavity velocity solutions displayed as midspliceplanes of velocity.}
\end{figure}

We also test using the 3D driven cavity benchmark problem, which is analogous to the 2D case. Here, the domain is the unit cube, there is no external forcing so $f=0$, homogeneous Dirichlet boundary conditions for velocity are enforced on the walls, and $[1, 0, 0]^T$ is enforced at the top of the box. The mesh is constructed by starting with Chebychev points on $[0, 1]$ to create a $ \mathcal{M}^3$ grid of rectangular boxes. Each box is first
split into 6 tetrahedra, and then each of these tetrahedra is split into 4 additional tetrahedra by a barycenter refinement / Alfeld split. Computations were done with $\mathcal{M}=$11 and 13, which provide for 796K and 1.3M dof respectively.  Results on the two meshes were very similar, so only the finest mesh results are given below.  For this test, $Re=\nu^{-1}$, and for our tests we use $Re = 1000$ and $2000$. The midspliceplane velocity vector solutions for $Re = 1000$ and $2000$ are shown in Figure \ref{cav3d}.

\subsubsection{Flow in a stenotic artery}

\begin{figure}[H]
\center
\includegraphics[width = .8\textwidth, height=.23\textwidth,viewport=0 0 1200 300, clip]{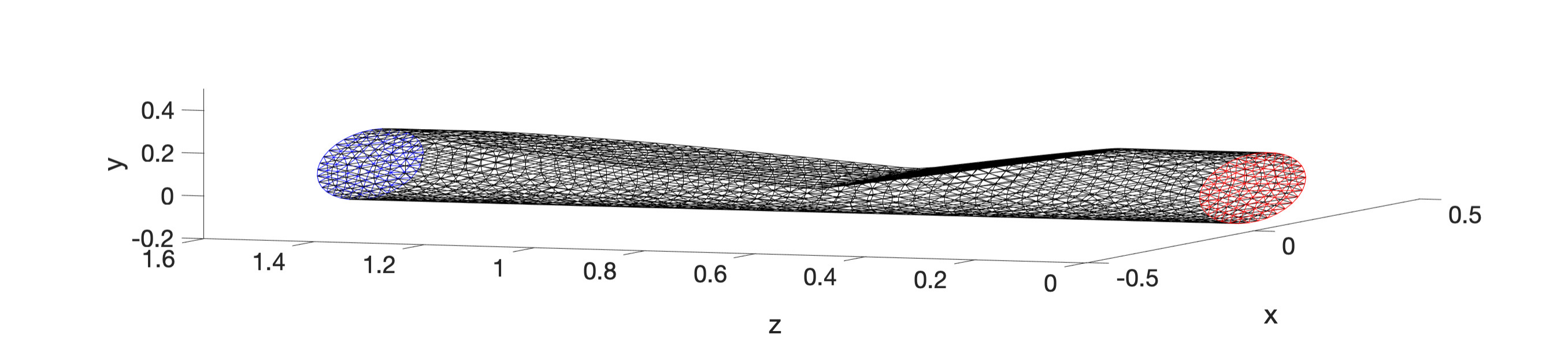}
\caption{\label{arterymesh} Shown above is the artery mesh (before the barycenter refinement is applied) restricted to the surface. }
\end{figure}

\begin{figure}[H]
\center
\includegraphics[width = .8\textwidth, height=.23\textwidth,viewport=50 0 1200 300, clip]{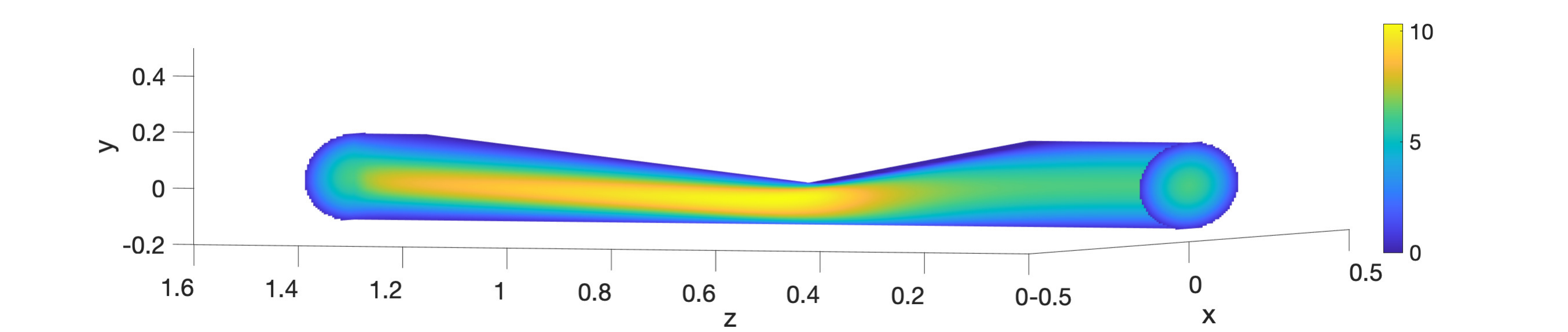}  
\caption{\label{artery100} Shown above are contour slices of the velocity magnitude for the $\nu=\frac{1}{500}$ solution found using Anderson accelerated Picard followed by Newton iterations until $H^1$ convergence of successive iterates to $10^{-10}$.}
\end{figure}

Our third test problem for CDA-Uzawa is for a stenotic artery model studied in \cite{TKLV24}.  The vessel domain is constructed as a 3D cylinder (radius $r=0.16cm$, length $l=1.6cm$) with significant deformation through compressions on the top of the vessel.  A plot of the domain is shown in Figure \ref{arterymesh} along with the (pre-barycenter refined) mesh (restricted to the surface) used for our computations with units in $cm$.  For boundary conditions, we use no-slip velocity on the artery walls and Dirichlet parabolic inflow (at $z=0$) and outflow (at $z=1.6$) with a maximum velocity at the vessel center of 6.22 cm/s (providing a flow rate of 15mL/minute).  The viscosity is taken as $\nu=0.002 g/(cm\cdot s)$.

The discretization is constructed as a barycenter refinement of a regular tetrahedralization of the domain, providing 52,912 total elements and 1.33M total dof when equipped with $(P_3,P_2^{disc})$ SV elements.  A plot of the solution found on this mesh using Anderson accelerated Picard (following \cite{PRX19,PR25}) is shown in Figure \ref{artery100}.

\subsubsection{2D channel flow past a block}

\begin{figure}[H]
  \centering
       \includegraphics[scale=0.11]{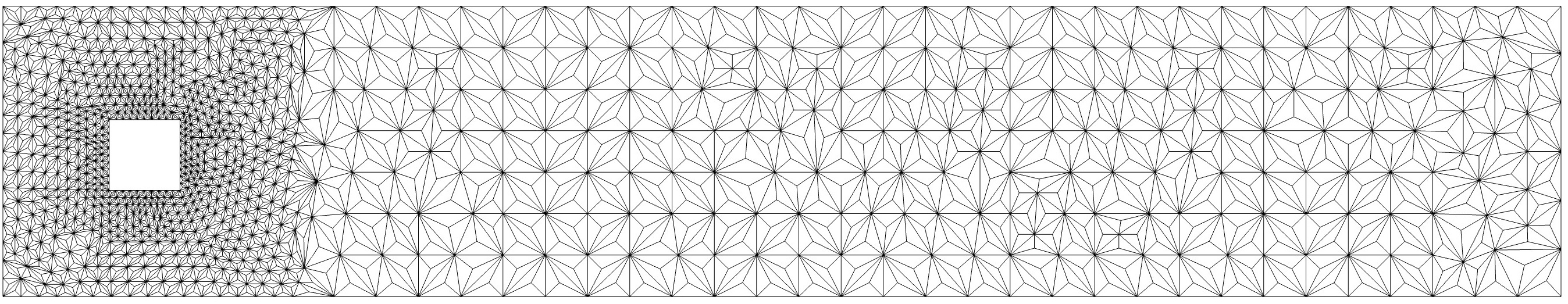}
  \caption{\label{cylpic} Shown above is a sample mesh for 2D channel flow past a block.}
\end{figure}

\begin{figure}[H]
  \centering
       \includegraphics[scale=0.34]{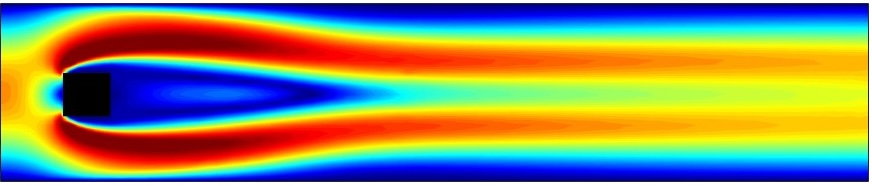} 
  \caption{\label{cylref} Shown above is a reference solution for 2D channel flow past a block with $Re$=100, shown as speed contours.}
\end{figure} 

Our next test problem setup is the channel flow past a block benchmark, which is also sometimes referred to as channel flow past a square cylinder \cite{ST96}. This problem has been widely studied both experimentally and numerically \cite{HRV24, R97, SDN99, TGO15}.  The computational domain is a rectangular channel of size $2.2\times0.41$ with a square obstacle of side length $0.1$ centered at $(0.2, 0.2)$ and with sides parallel to the channel.  We enforce no-slip velocity on the walls and block, and at the inflow and outflow, we enforce
\begin{align}
	u_1(0,y)=&u_1(2.2,y)=\frac{6}{0.41^2}y(0.41-y),\nonumber\\
	u_2(0,y)=&u_2(2.2,y)=0.
\end{align}

The external force is taken to be $f=0$. We use  $Re$=$\frac{UL}{\nu}$=100 and 150, which are calculated using the length scale $L=0.1$ as the width of the block and $U=1$, and then $\nu$ is chosen appropriately to produce these $Re$.  $(P_2, P_1^{disc})$ SV elements are used on a barycenter refined Delaunay mesh that provides 120K velocity and 89K pressure dof.  Figure \ref{cylpic} shows a mesh of the domain significantly coarser than what we use in the tests below.  Figure \ref{cylref} shows a $Re$=100 steady solution as speed contours.   We note that both $Re$=100 and 150 are known to produce periodic-in-time solutions in a time dependent code  \cite{HRV24}, but we are solving directly for steady state solutions at these $Re$.

\subsection{Convergence tests for CDA-Uzawa with noise-free partial solution data}

We now give results of the performance of CDA-Uzawa in the case when the partial solution data is exact.  We use four test problems: 2D driven cavity, 3D driven cavity, 2D channel flow past a block, and 3D flow in a stenotic artery model.

%\subsubsection{2D driven cavity}
For the 2D driven cavity, we test CDA-Uzawa convergence with varying $H$ and $\gamma$ and compare it to CDA-Picard convergence.  We first test the CDA-Uzawa scheme \eqref{CDAU1}-\eqref{CDAU2} for this test problem using $\gamma=1,\ 10,\ 100,\ 1000$, $Re$=5000 and 10000, and varying $H$.  For comparison, we also compute using the same parameters and CDA-Picard \cite{LHRV23} (the coupled scheme which is recovered from CDA-Uzawa if one replaces the two $p_k$'s in \eqref{CDAU1}-\eqref{CDAU2} by $p_{k+1}$'s).  

Results for $Re$=5000 are shown in Figure \ref{Uplots}, and we observe from the center plot that CDA-Uzawa is clearly accelerated by CDA; moreover, the smaller $H$ is, the faster the convergence.  Only for the finest $H$=1/64 was there a difference in results for $\gamma=1$ versus larger $\gamma$, which is in agreement with our theory that shows for $H$ sufficiently small, $\gamma$ must increase.  The plot at right in the figure shows convergence using $H$=1/64 and varying $\gamma$, and we observe the improvement of $\gamma\ge 10$ over $\gamma=1$. However, we also note that for $\gamma\ge 100$, convergence below $10^{-8}$ is never achieved (due to linear solver error introduced by larger condition numbers induced by larger $\gamma$ since the grad-div stabilization matrix is singular).  The plot at left in the figure shows convergence of CDA-Picard for varying $H$, and we observe that CDA-Picard convergence is nearly identical to CDA-Uzawa convergence using the same $H$ (provided $\gamma$ is chosen appropriately).  

\begin{figure}[H]
\center
CDA-Picard \hspace{1in} CDA-Uzawa \hspace{1in} CDA-Uzawa \\
\includegraphics[width = .3\textwidth, height=.26\textwidth,viewport=0 0 530 390, clip]{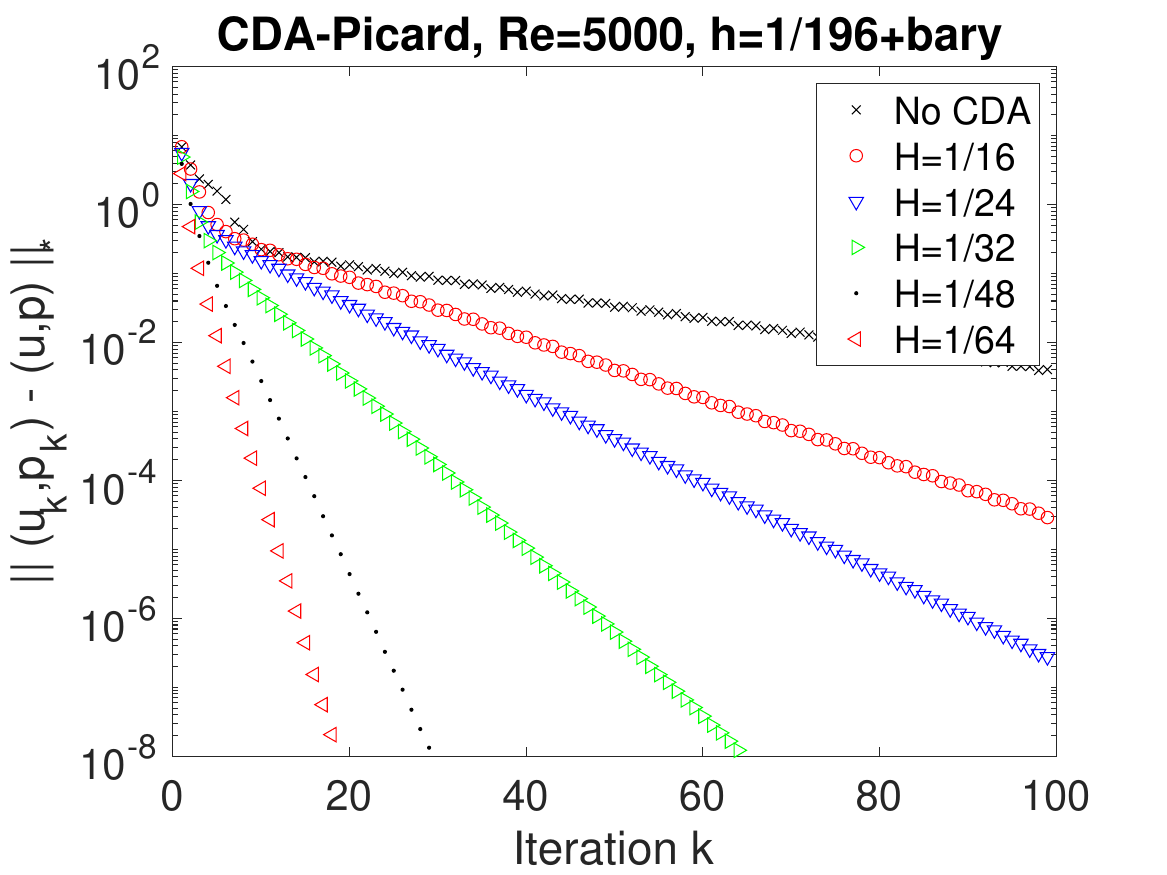}  
\includegraphics[width = .3\textwidth, height=.26\textwidth,viewport=0 0 530 390, clip]{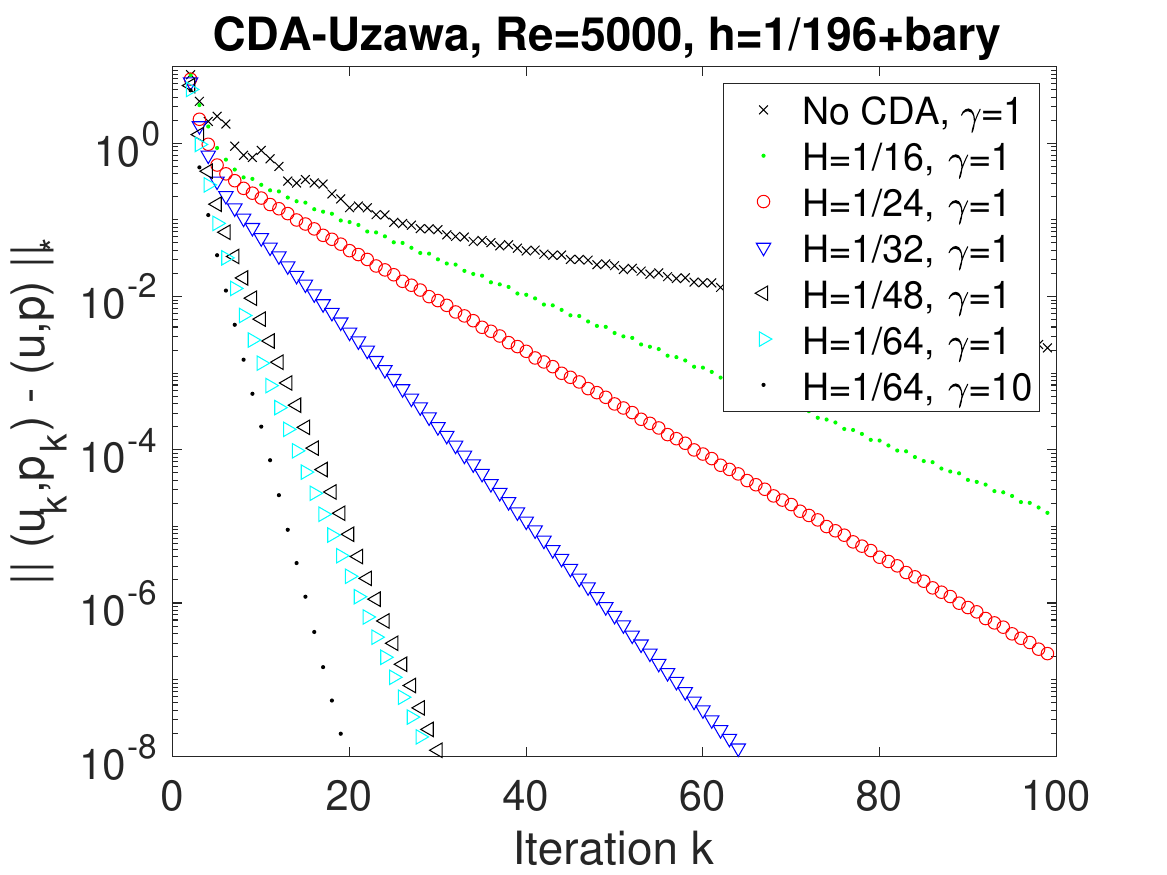}  
\includegraphics[width = .3\textwidth, height=.26\textwidth,viewport=0 0 530 390, clip]{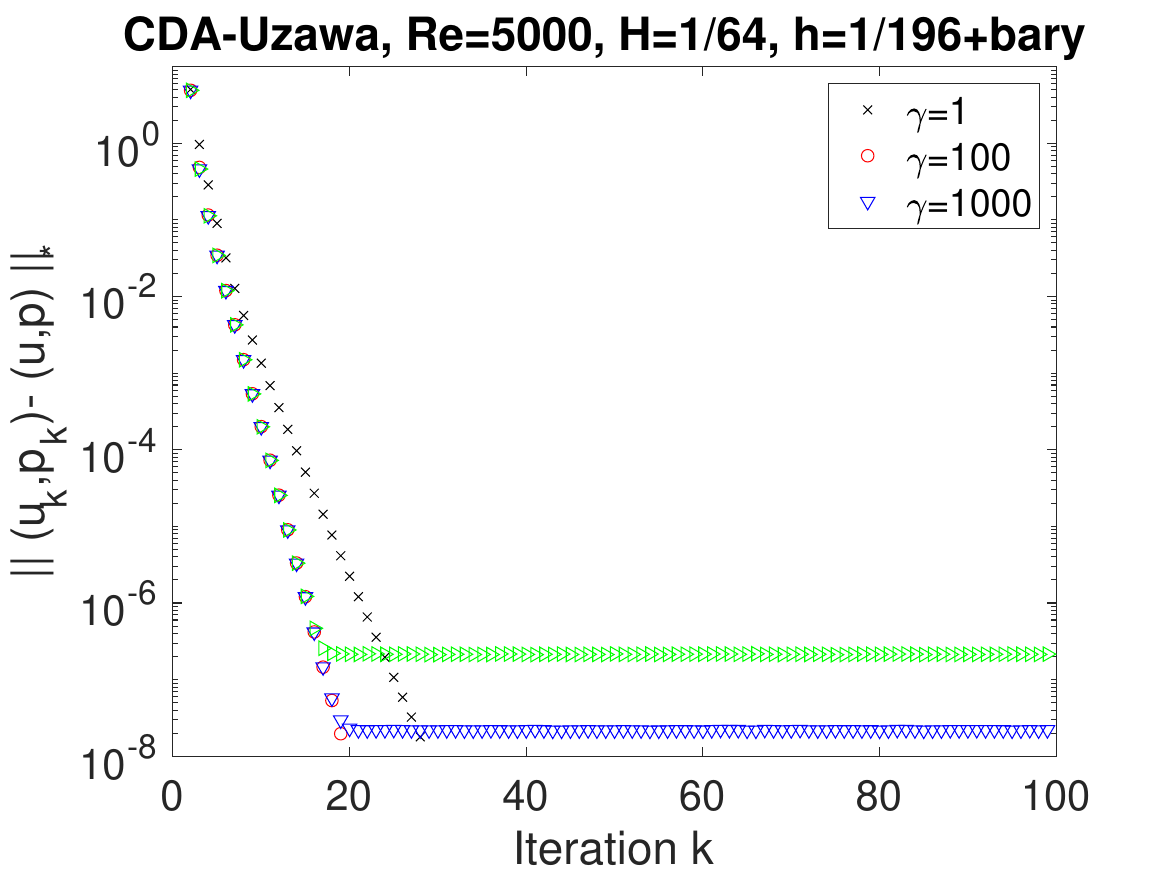}  
\caption{\label{Uplots} The plots above show convergence of (left) CDA-Picard with varying $H$, (center) CDA-Uzawa with varying $H$, and (right) CDA-Uzawa with varying $\gamma$ with $H$=1/64, all for $Re$=5000 2D driven cavity.   All tests use $(P_2,P_1^{disc})$ Scott-Vogelius elements on a barycenter refinement of a $h=\frac{1}{196}$ uniform mesh.}
\end{figure}

We repeat the 2D cavity test using $Re$=10000, and results are shown in Figure \ref{Uplots2}.  We observe that smaller $H$ is needed to achieve and accelerate convergence  (in agreement with the theory that $H$ needs to be smaller as $Re$ increases).  The key takeaway is that just as in the case of $Re=$5000, CDA-Uzawa performs very similar to CDA-Picard for the same $H$.

\begin{figure}[H]
\center
CDA-Picard \hspace{1in} CDA-Uzawa \hspace{1in} CDA-Uzawa \\
\includegraphics[width = .3\textwidth, height=.26\textwidth,viewport=0 0 530 390, clip]{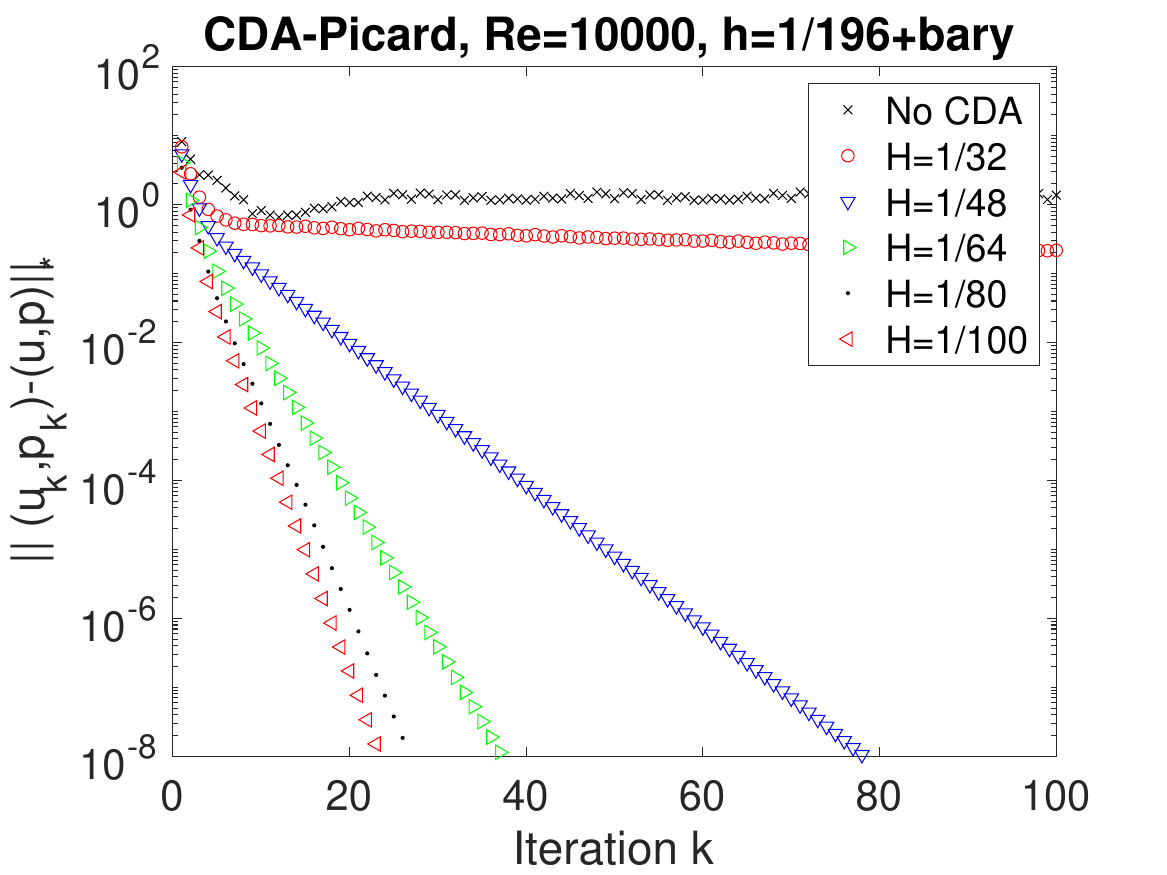}  
\includegraphics[width = .3\textwidth, height=.26\textwidth,viewport=0 0 530 390, clip]{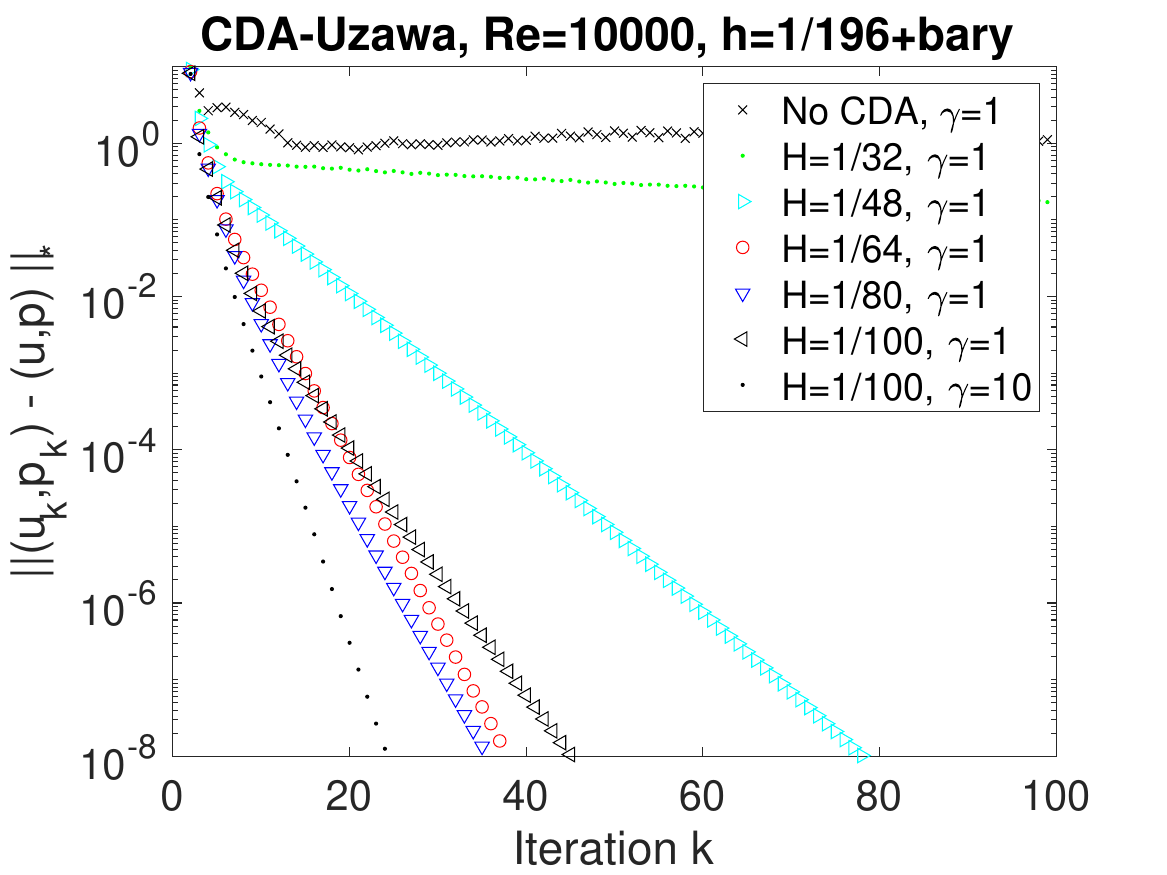}  
\includegraphics[width = .3\textwidth, height=.26\textwidth,viewport=0 0 535 390, clip]{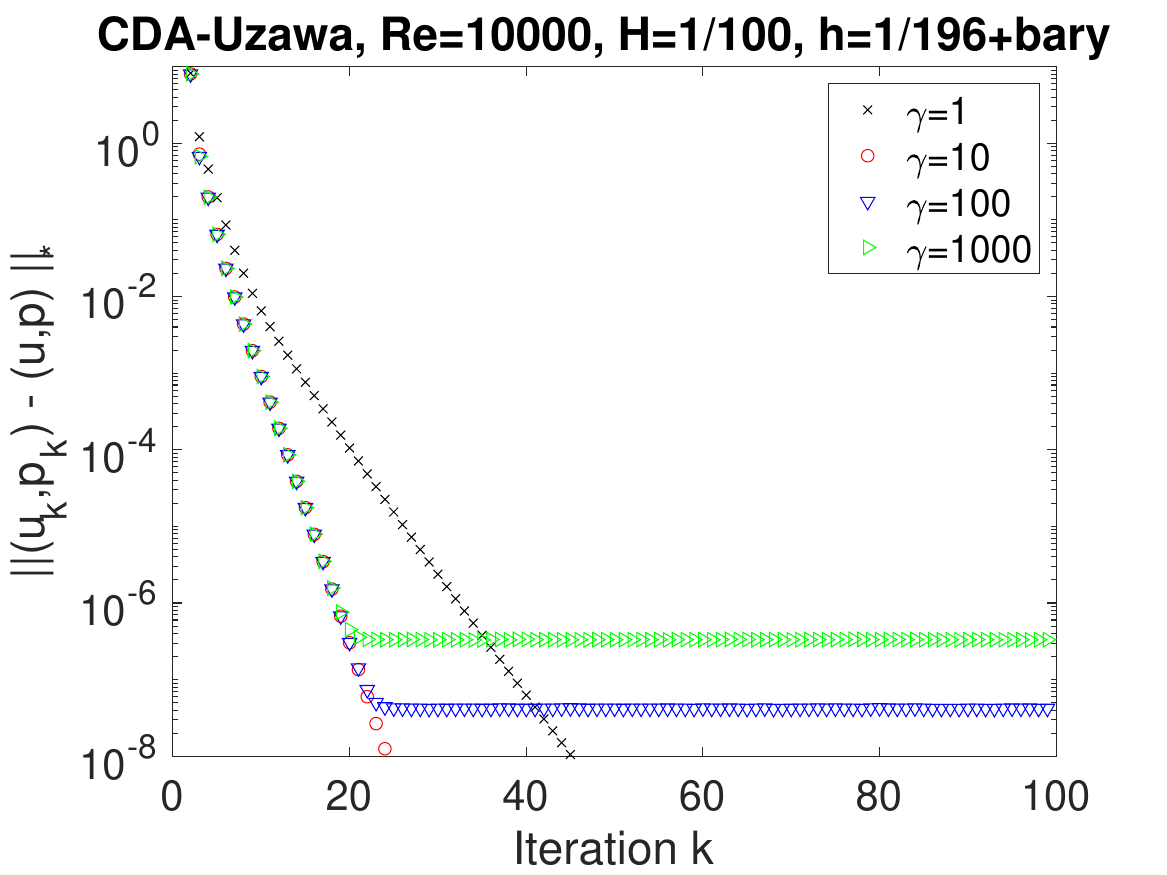}  
\caption{\label{Uplots2} The plots above show convergence of (left) CDA-Picard with varying $H$, (center) CDA-Uzawa with varying $H$, and (right) CDA-Uzawa with varying $\gamma$ with $H$=1/64, all for $Re$=10000 2D driven cavity.  All tests use $(P_2,P_1^{disc})$ Scott-Vogelius elements.}
\end{figure}

%\subsubsection{3D driven cavity}

%We next give results for CDA-Uzawa on the 3D driven cavity problem. 
Convergence results for CDA-Uzawa on the 3D driven cavity problem are shown in Figure \ref{cav3dconv}, and we observe that CDA-Uzawa is quite effective, even with relatively large $H$.  For both $Re$=1000 and 2000, there is a sharp cutoff of how small $H$ needs to be in order to achieve convergence.  

\begin{figure}[H]
\center
CDA-Uzawa $Re$=1000 \hspace{1in} CDA-Uzawa $Re$=2000\\
\includegraphics[width = .4\textwidth, height=.28\textwidth,viewport=0 0 530 390, clip]{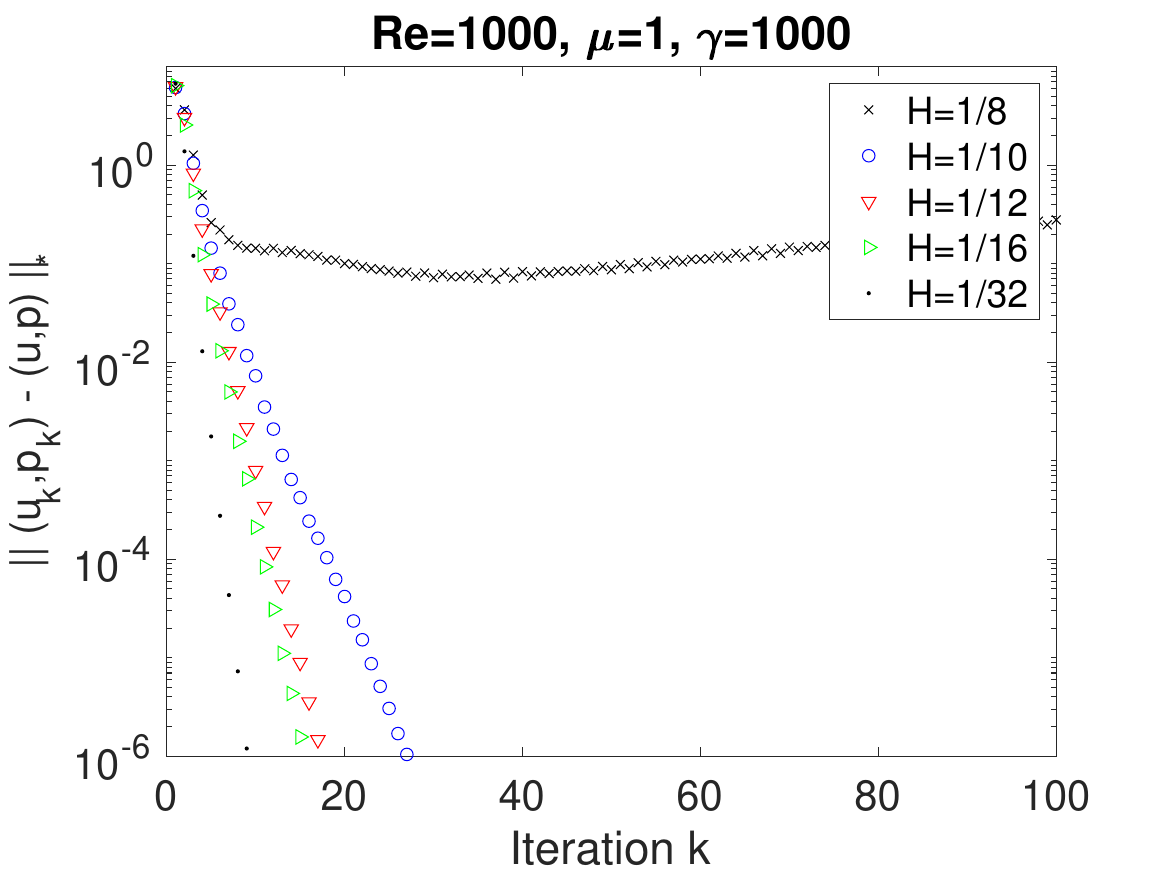}  
\includegraphics[width = .4\textwidth, height=.28\textwidth,viewport=0 0 530 390, clip]{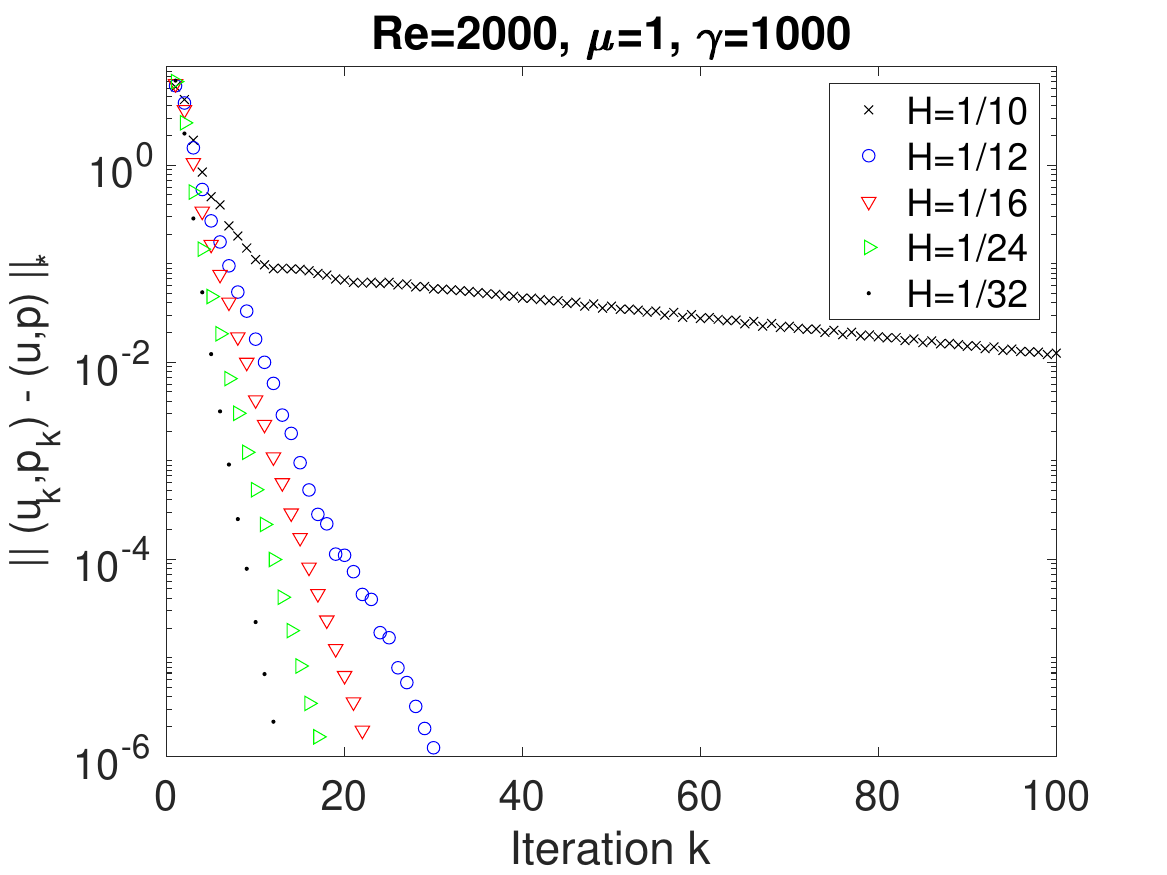}  
\caption{\label{cav3dconv} The plots above show convergence of CDA-Uzawa for (left) $Re$=1000 and (right) $Re$=2000 3D driven cavity flow.  All tests use $(P_3,P_2^{disc})$ SV elements.}
\end{figure}

%\subsubsection{2D channel flow past a block}
Convergence results for CDA-Uzawa on the 2D channel flow past a block are shown in Figure \ref{blockconv}.  For $Re$=100, we observe that Picard and Uzawa both converge at about the same rate without any CDA.  With CDA, especially with $H=2.2/24$ and smaller, convergence is significantly accelerated in CDA-Uzawa.  Interestingly, with large $H=2.2/4$ (just 16 total data measurement points), the convergence is slightly worse.  This does not contradict the theory because even though CDA transforms the Lipschitz constant from $\alpha$ to $O(H^{1/2} \alpha)$, there are constants involved in the latter, and so being slightly worse when $H$ is large is possible.  For $Re$=150, Uzawa and Picard fail, but CDA-Uzawa converges once $H$ is 2.2/20 or smaller.

Note that in these tests and those above, Uzawa and Picard give very similar behavior.  This is because of the particular Uzawa scheme we choose as well as the choice of spaces, which makes Uzawa equivalent to the iterated penalty Picard method.  

\begin{figure}[H]
\center
CDA-Uzawa $Re$=100 \hspace{1in} CDA-Uzawa $Re$=150\\
\includegraphics[width = .4\textwidth, height=.28\textwidth,viewport=0 0 530 390, clip]{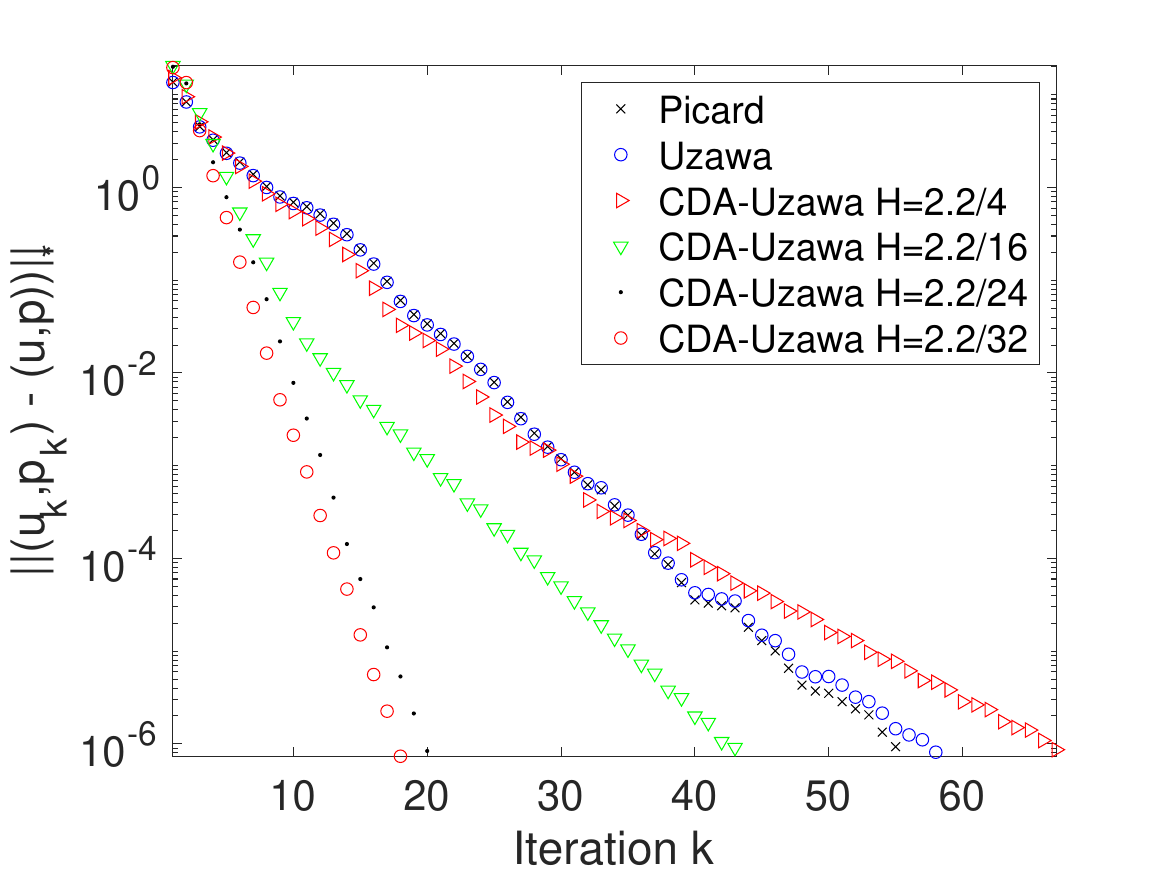}  
\includegraphics[width = .4\textwidth, height=.28\textwidth,viewport=0 0 530 390, clip]{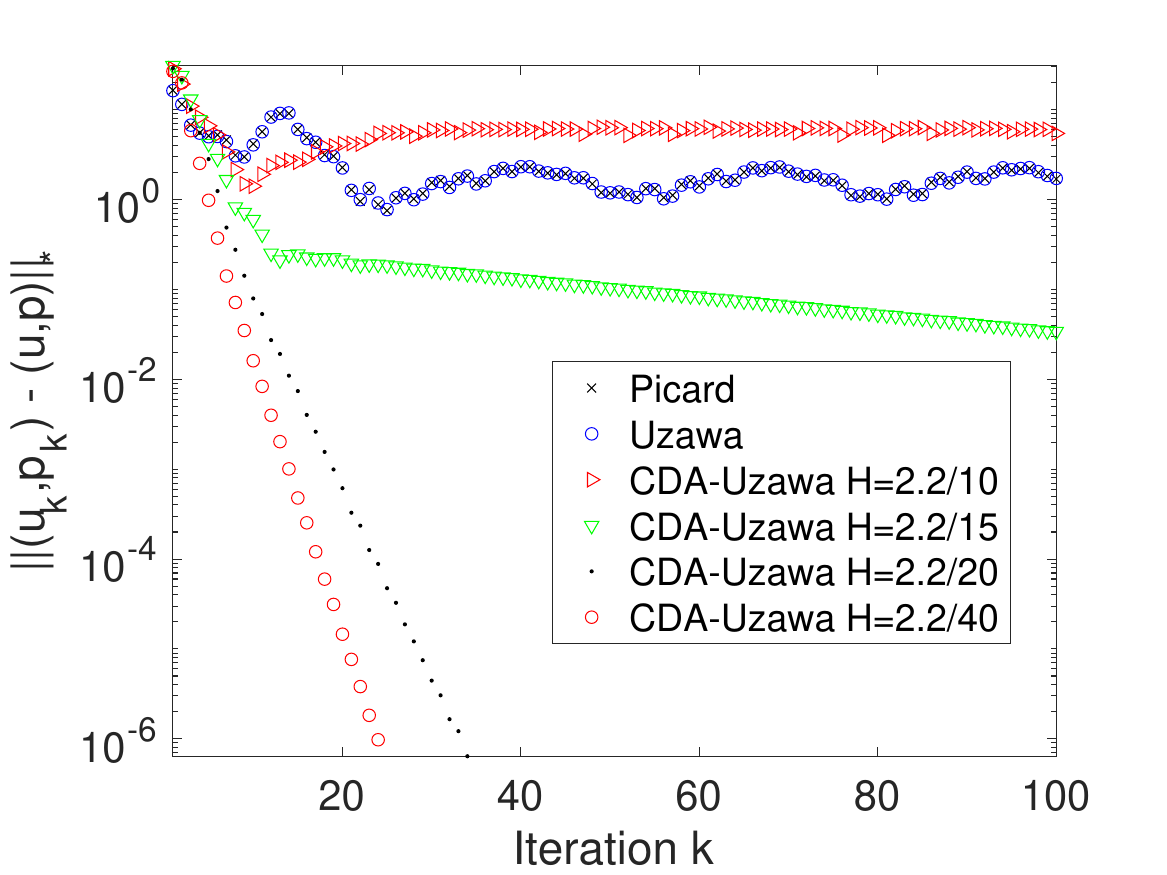}  
\caption{\label{blockconv} The plots above show convergence of 2D channel flow past a block with (left) $Re$=100 and (right) $Re$=150. }
\end{figure}

Our last noise-free test uses the stenotic artery problem.  Here, we tested several values of $H$ and no CDA.  First, we observe that the Picard solver without CDA converges in about 85 iterations (Uzawa without CDA converged about the same as Picard, plot omitted).  CDA-Uzawa accelerates convergence, and the more data, the faster the convergence.  Using $H=.32/5$ corresponds to knowing the solution at 583 nodes, $H=.32/50$ corresponds to 5201 nodes, and $H=.32/15$ corresponds to 12045 nodes.

\begin{figure}[H]
\center
CDA-Uzawa Artery\\
\includegraphics[width = .4\textwidth, height=.28\textwidth,viewport=0 0 530 390, clip]{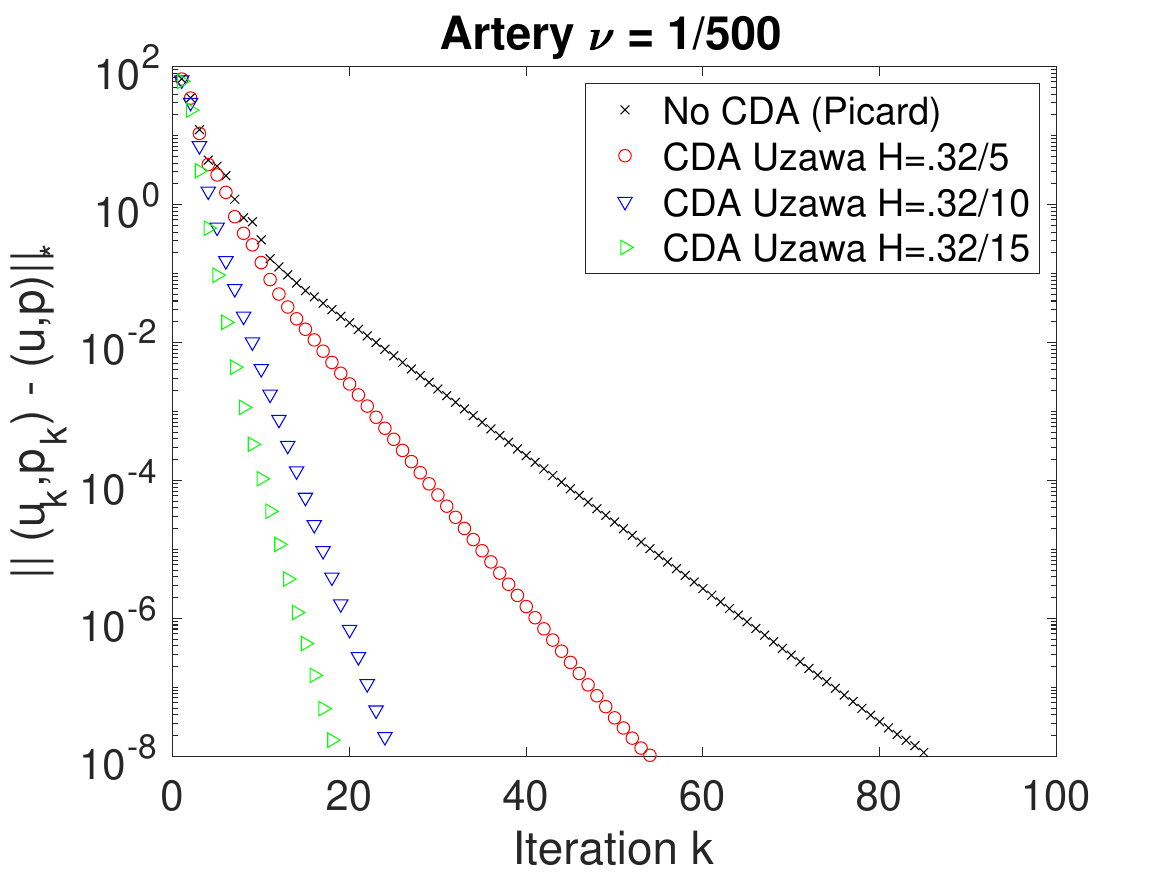}  
\caption{\label{arteryconv} Shown above is convergence of CDA-Uzawa for the stenotic artery problem with $\nu=\frac{1}{500}$ and varying $H$. }
\end{figure}

\subsection{Non-convergence of CDA-Uzawa using TH elements and noise-free partial solution data}

The theory of this paper assumes that $\nabla \cdot X_h\subseteq Q_h$, which holds for $(P_k,P_{k-1}^{disc})$ SV elements but not for $(P_k,P_{k-1})$ TH elements with $k\ge 2$.  Since $(P_2,P_1)$ TH elements are a very common choice in NSE computations, we now test CDA-Uzawa with this element choice.  We use the 2D driven cavity test with $Re$=5000, $h=1/64$ and $1/128$ barycenter refined uniform meshes, and we choose the grad-div stabilization parameter $\gamma=1$ for both CDA-Uzawa and for the true solution (the TH solution of the corresponding FEM NSE problem using grad-div stabilization).
Convergence results are shown in Figure \ref{Uplots3}.  Here, we observe linear convergence down to an error of about $10^{-4}$, where it bottoms out.  We note that it bottoms out slightly lower for the finer mesh (the error gets cut roughly in half).  Not coincidentally, the divergence error $\| \nabla \cdot u_h \| \approx 10^{-4}$ for these TH solutions. These plots show that the requirement of the theory $\nabla \cdot X_h\subseteq Q_h$ is necessary for the convergence in practice.

\begin{figure}[ht]
\center
CDA-Uzawa TH $h=\frac{1}{64}$\hspace{1in} CDA-Uzawa TH $h=\frac{1}{128}$ \\
\includegraphics[width = .4\textwidth, height=.26\textwidth,viewport=0 0 530 390, clip]{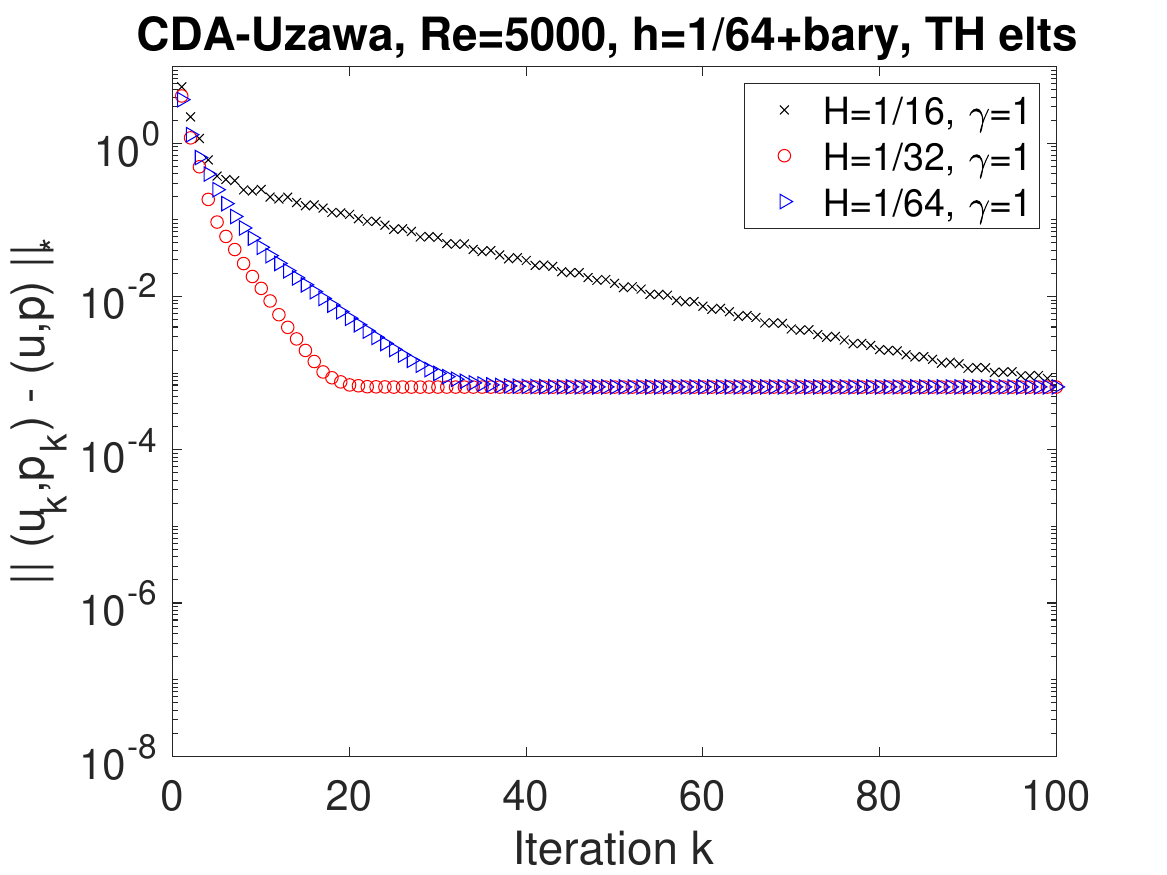}  
\includegraphics[width = .4\textwidth, height=.26\textwidth,viewport=0 0 530 390, clip]{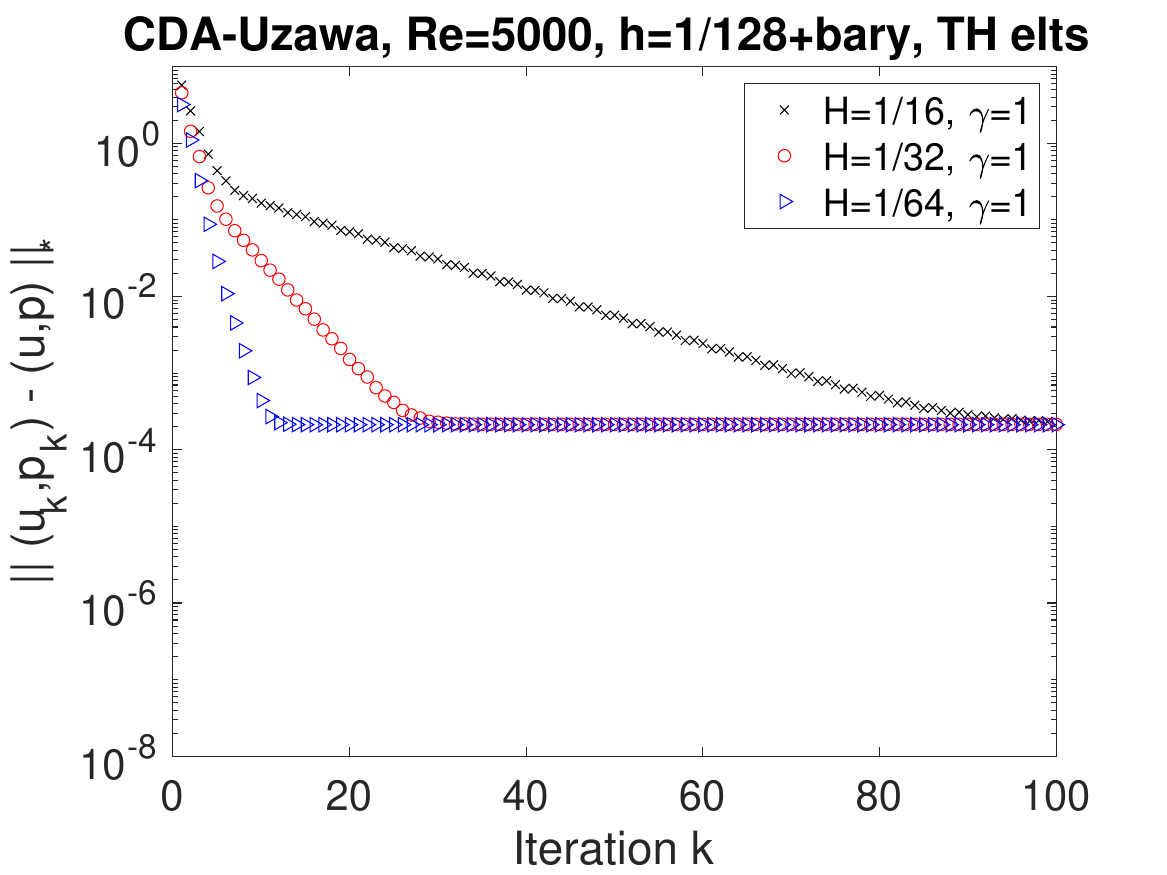}  
\caption{\label{Uplots3} The plots above show non-convergence of CDA-Uzawa using $(P_2,P_1)$ TH elements with $\gamma=1$ and varying $H$ on barycenter refined (left) $h$=1/64 and (right) $h$=1/128 meshes.}
\end{figure}

\subsection{Convergence of CDA-Uzawa with noisy partial solution data}

We now illustrate CDA-Uzawa convergence when the data contains noise. For our numerical tests, we consider CDA-Uzawa for the 2D and 3D driven cavity problems.  The noise is generated by first choosing a noise to signal ratio (NSR); we then modify the true velocity solution by adding component-wise uniform random numbers generated in $[-1, 1]$, which are scaled by the maximum of the true velocity solution and the NSR. 

The convergence results for the 2D driven cavity with $Re = 5000$, $H = 1/32$, $h = 1/128$, and varying NSR=0.05, 0.01 and 0.001, are shown in Figure \ref{2Dnoise} at left.  We observe that the difference in successive iterates $\| (u_k,p_k) - (u_{k-1},p_{k-1}) \|_*$ (which is also the nonlinear residual if we consider CDA-Uzawa as a fixed point iteration) converges linearly while the error $\| (u_k,p_k) - (u,p) \|_*$ converges linearly at first and then quickly bottoms at levels proportional to NSR, which agrees with our theory.  Note that the convergence was tested for other $h$ and $H$, and results showed analogous behavior, so we omit them here.

To avoid the large errors caused by noisy data, we also tested a strategy where we ran CDA-Uzawa until $\| (u_k,p_k) - (u_{k-1},p_{k-1}) \|_* < 10^{-4}$ and then switched to usual Newton (i.e., no CDA or other stabilization).  That is, we use CDA-Uzawa to generate a good initial guess for Newton, noting that usual Newton with an initial guess of zero will only converge  for $Re\le 2500$ \cite{PRTX25}.  This is observed to be a very effective strategy, even for the largest NSR=0.05.  

\begin{figure}[h!]
    \centering
    CDA-Uzawa \hspace{1.5in} CDA-Uzawa + Newton  \\
    \includegraphics[width = .4\textwidth, height=.28\textwidth,viewport=0 0 1200 870, clip]{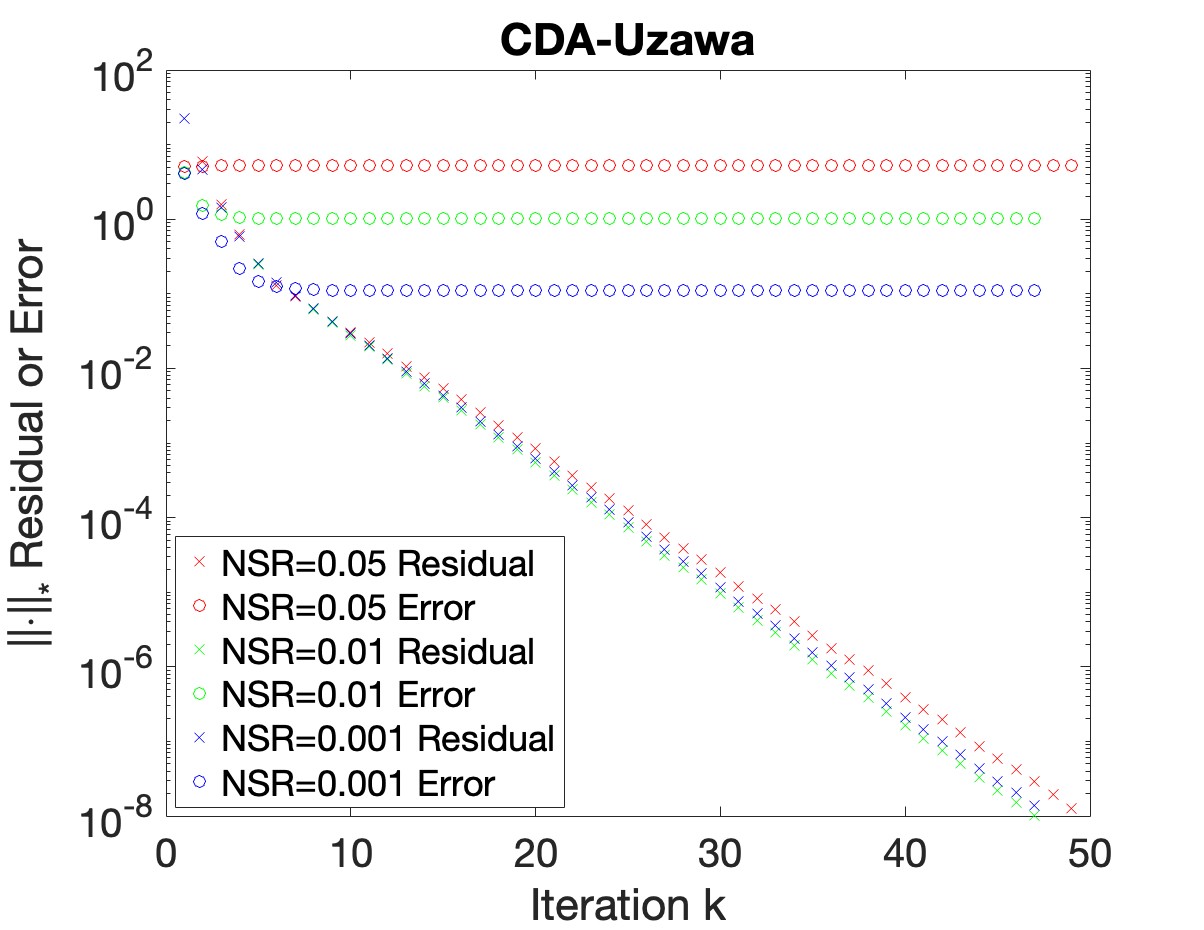}  
\includegraphics[width = .4\textwidth, height=.28\textwidth,viewport=0 0 1200 850, clip]{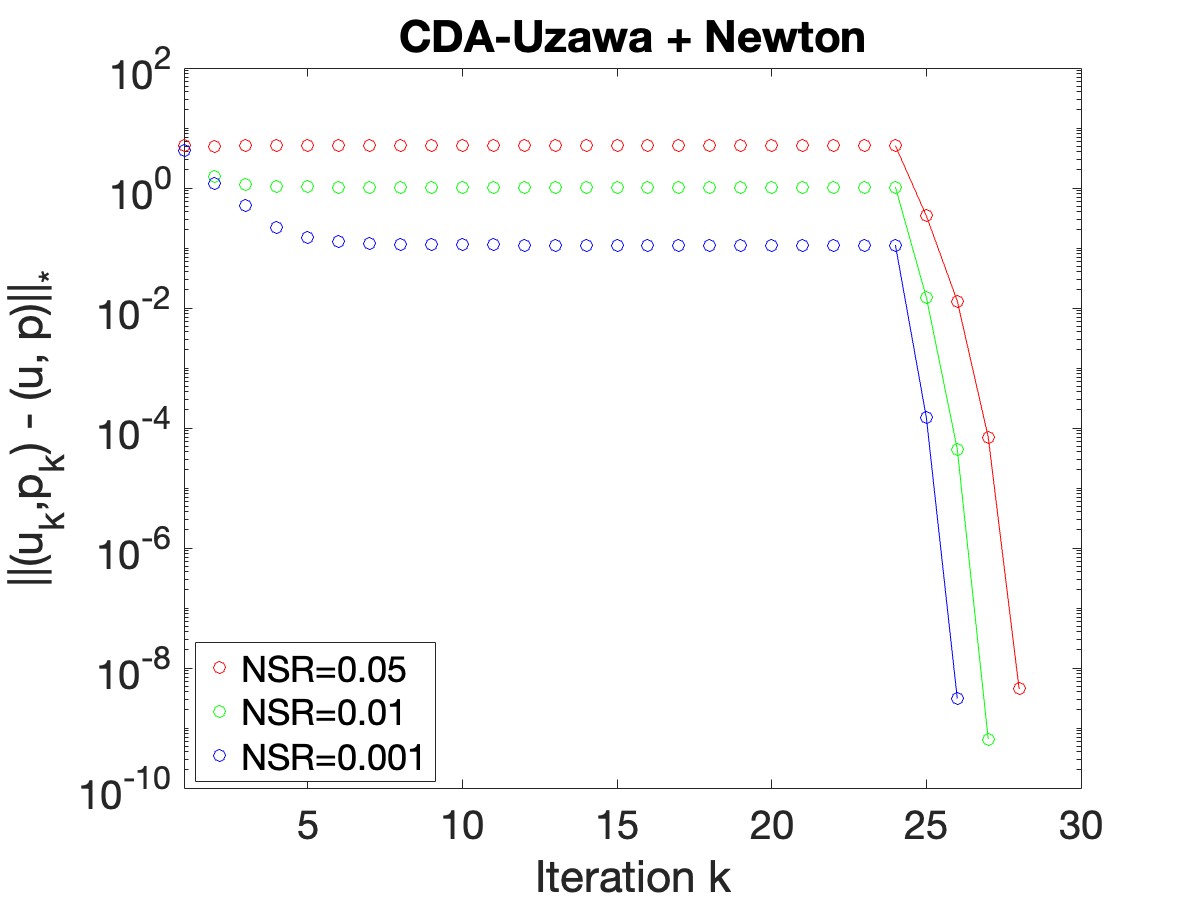}  
    \caption{ \label{2Dnoise} The plots above show the convergence of CDA-Uzawa with noisy partial solution data for $h=1/128$, $H = 1/32$, and $Re = 5000$. The plot on the left shows the residual and error of various NSR. The plot on the right illustrates the error convergence for various NSR once the residual is down to $10^{-4}$, and we switch to Newton.}
   
\end{figure}

%
%In order to try to improve the error, we also run CDA-Uzawa until the residual reaches $10^{-4}$. We then switch to Newton until the error reaches $10^{-8}$. The results of this experiment when $h = 1/128$ are also depicted in Figure \ref{2Dnoise} with the results for $h=1/64$ omitted. \\
%\indent The plot on the left shows both the $H^1$ residual convergence and $H^1$ error convergence of CDA-Uzawa with NSR$=0.05$, NSR$=0.01$, and NSR$=0.001$. We observe that the residuals of all three NSR values converge linearly down to $10^{-8}$ in about $60$ iterations. This is due to the fact that the noise is eliminated when taking the difference in successive iterates; therefore, these residual plots should look exactly like the numerical results illustrating Theorem \ref{thm1}. On the other hand, we see that the error bottoms out for all three NSR values. The error slightly improves for smaller NSR values, but no error reaches below $10^{-2}$. The plot on the right depicts the $H^1$ error for CDA-Uzawa with Newton for NSR$=0.05$, NSR$=0.01$, and NSR$=0.001$. We observe the error bottoming out for CDA-Uzawa with all three NSR values. Once the residual is down to $10^{-4}$, we switch to Newton, designated by the connected line. We observe immediate convergence to $10^{-8}$ (less than 10 iterations from when we start Newton) for all three NSR values.

We repeat this test for the 3D driven cavity with $Re$=1000 using $\mathcal{M} = 11$, $H= 1/32$, and $\gamma = 10$. We generate noise in the same way as for the 2D case and again use NSR values of $0.05, 0.01, 0.001$. The results for CDA-Uzawa and CDA-Uzawa with Newton are shown in Figure \ref{cav3dnoise}.  From Figure \ref{cav3dnoise}, we see results analogous to the 2D case in Figure \ref{2Dnoise}. For each NSR value, the residual $\| (u_k,p_k) - (u_{k-1},p_{k-1}) \|_*$ converges linearly, while the error $\| (u_k,p_k) - (u,p) \|_*$ bottoms out well above $10^{-2}$. Once Newton is applied, the error quickly converges for each NSR value.  We note that (unstabilized) Newton with an initial guess of 0 for this test problem will converge only up to $Re=$200 \cite{PRTX25}.

\begin{figure}[H]
    \centering
        CDA-Uzawa \hspace{1.5in} CDA-Uzawa + Newton  \\
    \includegraphics[width=0.4\linewidth, height=.28\textwidth,viewport=0 0 1300 1060, clip]{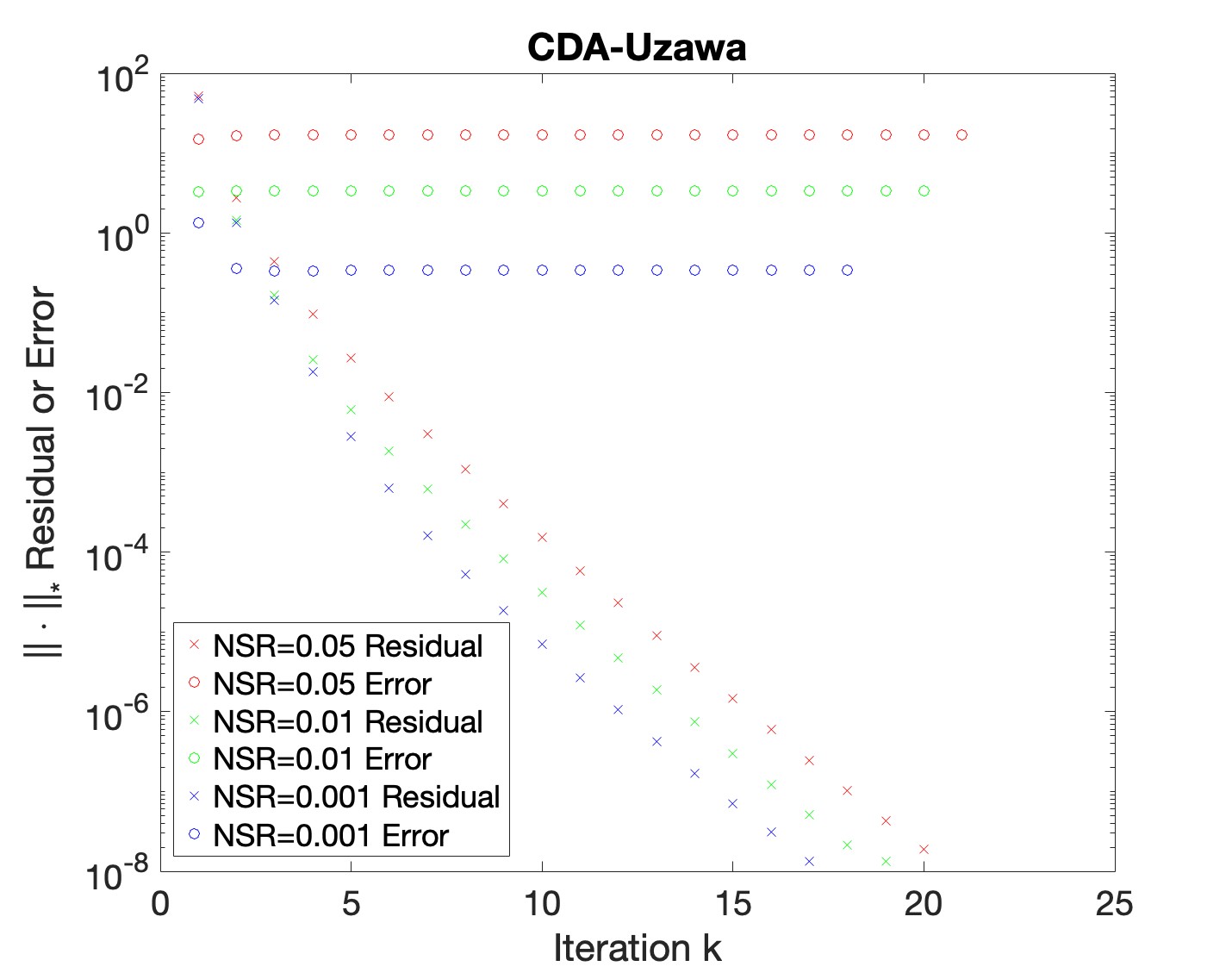}
     \includegraphics[width=0.4\linewidth, height=.28\textwidth,viewport=0 0 1100 840, clip]{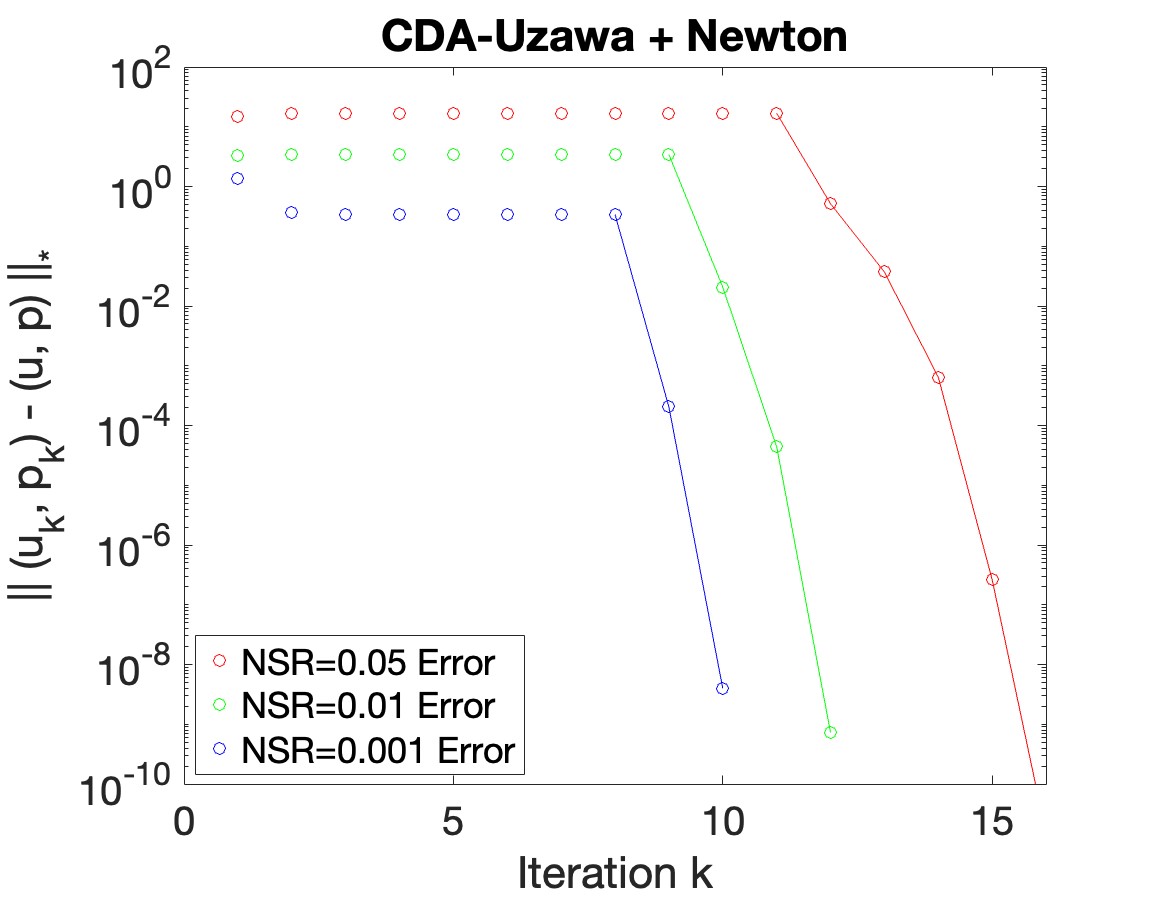}
    \caption{\label{cav3dnoise} The plots above show the CDA-Uzawa convergence with noisy partial solution data without Newton (left) and with Newton (right) for $\mathcal{M} = 11$, $H = 1/32$, $\gamma = 10$, and $Re = 1000$ for various NSR.}
    
\end{figure}

\section{Conclusions}

We proposed, analyzed and tested a CDA-Uzawa nonlinear solver  in order to create a significantly more efficient alternative to CDA-Picard for solving the NSE in the setting where partial solution data is available.  While Uzawa type methods are known to be much more efficient than Picard iterations due to the easier linear systems needing solved, they are also known to converge slowly and not be robust with respect to $Re$.  We have shown analytically and numerically that the proposed CDA-Uzawa solver, due to being equipped with CDA, grad-div stabilization, and particular choice of finite element spaces, has essentially no loss in convergence rate compared to CDA-Picard. 

\section{Acknowledgements}
The authors thank Dr. Jimena Tempestti and Prof. Alessandro Veneziani for providing a mesh and assistance with the stenotic artery test problem.

\section{Data Availability}

Data will be made available on request.

\section{Declaration of competing interest}

The authors declare the following financial interests / personal relationships which may be considered as competing interests: All authors report financial support from the Department of Energy, grant DE-SC0025292.

\bibliographystyle{plain}
%\bibliography{graddiv}

\begin{thebibliography}{10}

\bibitem{arnold:qin:scott:vogelius:2D}
D.~Arnold and J.~Qin.
\newblock Quadratic velocity/linear pressure {S}tokes elements.
\newblock In R.~Vichnevetsky, D.~Knight, and G.~Richter, editors, {\em Advances in Computer Methods for Partial Differential Equations VII}, pages 28--34. IMACS, 1992.

\bibitem{AOT14}
A.~Azouani, E.~Olson, and E.~S. Titi.
\newblock Continuous data assimilation using general interpolant observables.
\newblock {\em Journal of Nonlinear Science}, 24:277--304, 2014.

\bibitem{AT14}
A.~Azouani and E.~Titi.
\newblock Feedback control of nonlinear dissipative systems by finite determining parameters - a reaction-diffusion paradigm.
\newblock {\em Evolution Equations and Control Theory}, 3(4):579--594, 2014.

\bibitem{benzi}
M.~Benzi and M.~Olshanskii.
\newblock An augmented {L}agrangian-based approach to the {O}seen problem.
\newblock {\em SIAM J. Sci. Comput.}, 28:2095--2113, 2006.

\bibitem{BGHRR25}
C.~Bernardi, V.~Girault, F.~Hecht, P.-A. Raviart, and B.~Riviere.
\newblock {\em Mathematics and Finite Element Discretizations of Incompressible {N}avier-{S}tokes Flows}.
\newblock Society for Industrial and Applied Mathematics, Philadelphia, PA, 2024.

\bibitem{ARR_2022}
A.~Biswas, K.~R. Brown, and V.~R. Martinez.
\newblock Mesh-free interpolant observables for continuous data assimilation.
\newblock {\em Ann. Appl. Math.}, 38(3):296--355, 2022.

\bibitem{AJS_2022}
A.~Biswas, J.~Tian, and S.~Ulusoy.
\newblock Error estimates for deep learning methods in fluid dynamics.
\newblock {\em Numer. Math.}, 151(3):753--777, 2022.

\bibitem{BB12}
Steffen B{\"o}rm and Sabine Le~Borne.
\newblock {$\mathcal H$}-{LU} factorization in preconditioners for augmented {L}agrangian and grad-div stabilized saddle point systems.
\newblock {\em Internat. J. Numer. Methods Fluids}, 68(1):83--98, 2012.

\bibitem{BS08}
S.C. Brenner and L.~R. Scott.
\newblock {\em The Mathematical Theory of Finite Element Methods}, volume~15 of {\em Texts in Applied Mathematics}.
\newblock Springer Science+Business Media, LLC, 2008.

\bibitem{CH22}
E.~Carlson, J.~Hudson, A.~Larios, V.~R. Martinez, E.~Ng, and J.~P. Whitehead.
\newblock Dynamically learning the parameters of a chaotic system using partial observations.
\newblock {\em DCDS}, 32(8):3809--3839, 2022.

\bibitem{CHS17}
P.~Chen, J.~Huang, and H.~Sheng.
\newblock Solving steady incompressible {N}avier-{S}tokes equations by the {A}rrow-{H}urwicz method.
\newblock {\em Journal of Computational and Applied Mathematics}, 311:100--114, 2017.

\bibitem{CHS15}
P.~Chen, J.~Huang P, and H.~Sheng.
\newblock Some {U}zawa methods for steady incompressible {N}avier-{S}tokes equations discretized by mixed element methods.
\newblock {\em J. Comput. Appl. Math.}, 273:313--325, 2015.

\bibitem{CFLS25}
A.~Cibik, R.~Fang, W.~Layton, and F.~Siddiqua.
\newblock Adaptive parameter selection in nudging based data assimilation.
\newblock {\em Computer Methods in Applied Mechanics and Engineering}, 433(117526), 2025.

\bibitem{C93}
R.~Codina.
\newblock {An iterative penalty method for the finite element solution of the stationary Navier-Stokes equations}.
\newblock {\em Computer Methods in Applied Mechanics and Engineering}, 110:237--262, 1993.

\bibitem{ECG05}
E.~Erturk, T.~C. Corke, and C.~G\"okc\"ol.
\newblock Numerical solutions of {2D}-steady incompressible driven cavity flow at high {R}eynolds numbers.
\newblock {\em Int. J. Numer. Methods Fluids}, 48:747--774, 2005.

\bibitem{FLMW24}
A.~Farhat, A.~Larios, V.~Martinez, and J.~Whitehead.
\newblock Identifying the body force from partial observations of a two-dimensional incompressible velocity field.
\newblock {\em Physical Review Letters}, 9:054602, 2024.

\bibitem{FRV25}
V.~Fisher, L.~Rebholz, and D.~Vargun.
\newblock An efficient splitting iteration for a {CDA}-accelerated solver for incompressible flow problems.
\newblock {\em Submitted}, 2025.

\bibitem{GLNR24}
B.~Garcia-Archilla, X.~Li, J.~Novo, and L.~Rebholz.
\newblock Enhancing nonlinear solvers for the {N}avier-{S}tokes equations with continuous (noisy) data assimilation.
\newblock {\em Computer Methods in Applied Mechanics and Engineering}, 424(116903):1--15, 2024.

\bibitem{GN20}
B.~Garcia-Archilla and J.~Novo.
\newblock Error analysis of fully discrete mixed finite element data assimilation schemes for the {N}avier-{S}tokes equations.
\newblock {\em Advances in Computational Mathematics}, pages 46--61, 2020.

\bibitem{GRVZ23}
P.~Guven Geredeli, L.~Rebholz, D.~Vargun, and A.~Zytoon.
\newblock Improved convergence of the {A}rrow-{H}urwicz iteration for the {N}avier-{S}tokes equation via grad-div stabilization and {A}nderson acceleration.
\newblock {\em Journal of Computational and Applied Mathematics}, 422(114920):1--16, 2023.

\bibitem{GR86}
V.~Girault and P.-A. Raviart.
\newblock {\em Finite element methods for Navier-Stokes equations: Theory and Algorithms}.
\newblock Springer-Verlag, 1986.

\bibitem{G89b}
M.~Gunzburger.
\newblock {Iterative penalty methods for the Stokes and Navier-Stokes equations}.
\newblock {\em Proceedings from Finite Element Analysis in Fluids conference, University of Alabama, Huntsville}, pages 1040--1045, 1989.

\bibitem{GS19}
J.~Guzman and L.~R. Scott.
\newblock The {Scott-Vogelius} finite elements revisited.
\newblock {\em Math. Comp.}, 88(316):515--529, 2019.

\bibitem{HRV24}
E.~Hawkins, L.~Rebholz, and D.~Vargun.
\newblock Removing splitting/modeling error in projection/penalty methods for {N}avier-{S}tokes simulations with continuous data assimilation.
\newblock {\em Communications in Mathematical Research}, 40(1):1--29, 2024.

\bibitem{HR13}
T.~Heister and G.~Rapin.
\newblock Efficient augmented {L}agrangian-type preconditioning for the {O}seen problem using grad-div stabilization.
\newblock {\em Int. J. Numer. Meth. Fluids}, 71:118--134, 2013.

\bibitem{JLMNR17}
V.~John, A.~Linke, C.~Merdon, M.~Neilan, and L.~G. Rebholz.
\newblock On the divergence constraint in mixed finite element methods for incompressible flows.
\newblock {\em SIAM Review}, 59(3):492--544, 2017.

\bibitem{JP23}
M.S. Jolly and A.~Pakzad.
\newblock Data assimilation with higher order finite element interpolants.
\newblock {\em International Journal for Numerical Methods in Fluids}, 95:472--490, 2023.

\bibitem{LP24}
A.~Larios and Y.~Pei.
\newblock Nonlinear continuous data assimilation.
\newblock {\em Evolution Equations and Control Theory}, 13:329--348, 2024.

\bibitem{Laytonbook}
W.~Layton.
\newblock {\em An {I}ntroduction to the {N}umerical {A}nalysis of {V}iscous {I}ncompressible {F}lows}.
\newblock SIAM, Philadelphia, 2008.

\bibitem{LHRV23}
X.~Li, E.~Hawkins, L.~Rebholz, and D.~Vargun.
\newblock Accelerating and enabling convergence of nonlinear solvers for {N}avier-{S}tokes equations by continuous data assimilation.
\newblock {\em Computer Methods in Applied Mechanics and Engineering}, 416:1--17, 2023.

\bibitem{MS18}
H.~Morgan and L.R. Scott.
\newblock Towards a unified finite element method for the {S}tokes equations.
\newblock {\em SIAM Journal on Scientific Computing}, 40(1):A130--A141, 2018.

\bibitem{PR25}
S.~Pollock and L.~Rebholz.
\newblock {\em Anderson Acceleration for Numerical PDEs}.
\newblock Society for Industrial and Applied Mathematics, Philadelphia, PA, 2025.

\bibitem{PRTX25}
S.~Pollock, L.~Rebholz, X.~Tu, and M.~Xiao.
\newblock Analysis of the {P}icard-{N}ewton iteration for the {N}avier-{S}tokes equations: global stability and quadratic convergence.
\newblock {\em Journal of Scientific Computing}, 14:25:1--23, 2025.

\bibitem{PRX19}
S.~Pollock, L.~Rebholz, and M.~Xiao.
\newblock Anderson-accelerated convergence of {P}icard iterations for incompressible {N}avier-{S}tokes equations.
\newblock {\em SIAM Journal on Numerical Analysis}, 57:615-- 637, 2019.

\bibitem{RZ21}
L.~G. Rebholz and C.~Zerfas.
\newblock Simple and efficient continuous data assimilation of evolution equations via algebraic nudging.
\newblock {\em Numerical Methods for Partial Differential Equations}, 37(3):2588--2612, 2021.

\bibitem{R97}
W.~Rodi.
\newblock Comparison of {LES} and {RANS} calculations of the flow around bluff bodies.
\newblock {\em Journal of Wind Engineering and Industrial Aerodynamics}, 69-71:55--75, 1997.
\newblock Proceedings of the 3rd International Colloqium on Bluff Body Aerodynamics and Applications.

\bibitem{ST96}
M.~Sch$\ddot{\mbox{a}}$fer and S.~Turek.
\newblock The benchmark problem `flow around a cylinder' flow simulation with high performance computers {II}.
\newblock {\em in E.H. Hirschel (Ed.), Notes on Numerical Fluid Mechanics}, 52, Braunschweig, Vieweg:547--566, 1996.

\bibitem{scott:vogelius:conforming}
L.~R. Scott and M.~Vogelius.
\newblock Conforming finite element methods for incompressible and nearly incompressible continua.
\newblock In {\em Large-Scale Computations in Fluid Mechanics, Part 2}, volume 22-2 of {\em Lectures in Applied Mathematics}, pages 221--244. Amer. Math. Soc., 1985.

\bibitem{SZ90}
L.~R. Scott and S.~Zhang.
\newblock Finite element interpolation of nonsmooth functions satisfying boundary conditions.
\newblock {\em Math. Comp.}, 54(190):483--493, 1990.

\bibitem{SDN99}
A.~Sohankar, L.~Davidson, and C.~Norberg.
\newblock Large eddy simulation of flow past a square cylinder: Comparison of different subgrid scale models.
\newblock {\em Journal of Fluids Engineering}, 122(1):39--47, 11 1999.

\bibitem{TCEK23}
A.~Takhirov, A.~Cibik, F.~Eroglu, and S.~Kaya.
\newblock An improved {A}rrow-{H}urwicz method for the steady-state {N}avier-{S}tokes equations.
\newblock {\em Journal of Scientific Computing}, 96(52):1--21, 2023.

\bibitem{temam}
R.~Temam.
\newblock {\em {N}avier-{S}tokes equations}.
\newblock Elsevier, North-Holland, 1991.

\bibitem{TKLV24}
J.~Martín Tempestti, S.~Kim, B.~Lindsey, and A.~Veneziani.
\newblock A pseudo-spectral method for wall shear stress estimation from {D}oppler ultrasound imaging in coronary arteries.
\newblock {\em Cardiovascular Engineering and Technology}, 15(6):647--666, 2024.

\bibitem{TGO15}
F.X. Trias, A.~Gorobets, and A.~Oliva.
\newblock Turbulent flow around a square cylinder at {R}eynolds number 22,000: A {DNS} study.
\newblock {\em Computers \& Fluids}, 123:87--98, 2015.

\bibitem{zhang:scott:vogelius:3D}
S.~Zhang.
\newblock A new family of stable mixed finite elements for the {3D} {S}tokes equations.
\newblock {\em Math. Comp.}, 74(250):543--554, 2005.

\end{thebibliography}
%\input{output.bbl}

\end{document}